\documentclass[a4paper]{amsart}

\usepackage{amsmath}
\usepackage{dsfont}
\usepackage{color}
\usepackage{cancel}
\usepackage{amssymb}
\usepackage[mathscr]{eucal}

\numberwithin{equation}{section}

\newtheorem{theorem}{Theorem}[section]
\newtheorem{definition}[theorem]{Definition}
\newtheorem{lemma}[theorem]{Lemma}
\newtheorem{remark}[theorem]{Remark}

\newtheorem{proposition}[theorem]{Proposition}

\newtheorem{example}[theorem]{Example}

\newcommand{\N}{\mathbb{N}}
\newcommand{\R}{\mathbb{R}}
\newcommand{\E}{\mathbb{E}}
\newcommand{\mP}{\mathbb{P}}

\newcommand{\T}{\mathbb{T}}
\newcommand{\Z}{\mathbb{Z}}

\newcommand{\Rq}{\overline{\R}}
\newcommand{\1}{\mathds{1}}
\newcommand{\cA}{\mathcal{A}}
\newcommand{\cR}{\mathcal{R}}

\newcommand{\cF}{\mathcal{F}}

\newcommand{\cN}{\mathcal{N}}

\newcommand{\cB}{\mathcal{B}}

\newcommand{\cS}{\mathcal{S}}

\newcommand{\cU}{\mathcal{U}}
\newcommand{\cW}{\mathcal{W}}

\newcommand{\id}{\hbox{id}}

\newcommand{\eps}{\varepsilon}

\newcommand{\abvallr}[1]{\left|#1\right|}
\newcommand{\abvalb}[1]{\big|#1\big|}

\newcommand{\clintb}[1]{\big[#1\big]}

\newcommand{\opintb}[1]{\big(#1\big)}

\newcommand{\opintBg}[1]{\Bigg(#1\Bigg)}
\newcommand{\opintlr}[1]{\left(#1\right)}
\newcommand{\set}[1]{\{#1\}}
\newcommand{\setb}[1]{\big\{#1\big\}}
\newcommand{\setB}[1]{\Big\{#1\Big\}}

\newcommand{\norm}[1]{\|#1\|}
\newcommand{\normb}[1]{\big\|#1\big\|}

\newcommand{\minequ}[1]{\!\!\!&#1&\!\!\!}
\newcommand{\vp} {{\varphi}}
\newcommand{\fa} {\quad \text{for all }\,}
\newcommand{\fs} {\quad \text{for some }\,}
\newcommand{\faa} {\quad \text{for almost all }\,}
\newcommand{\rk}{\operatorname{rk}}
\newcommand{\dist}{\operatorname{dist}}

\newcommand{\rmd}{\mathrm{d}}

\newcommand{\defref}[1]{Definition~\ref{#1}}
\newcommand{\examref}[1]{Example~\ref{#1}}

\newcommand{\lemmaref}[1]{Lemma~\ref{#1}}
\newcommand{\propref}[1]{Proposition~\ref{#1}}

\newcommand{\secref}[1]{Section~\ref{#1}}
\newcommand{\subsecref}[1]{Subsection~\ref{#1}}
\newcommand{\theoref}[1]{Theorem~\ref{#1}}

\newcommand{\mand}{\text{ and }}
\newcommand{\mor}{\text{ or }}
\newcommand{\mwith}{\text{ with }}
\newcommand{\qandq}{\quad\text{and}\quad}
\newcommand{\qorq}{\quad\text{or}\quad}

\DeclareMathOperator*{\esssup}{ess\,sup}
\DeclareMathOperator*{\essinf}{ess\,inf}

\begin{document}

\title[The dichotomy spectrum and pitchfork bifurcations with additive noise]{The dichotomy spectrum for random dynamical systems\\and pitchfork bifurcations with additive noise}

\author[M.~Callaway]{Mark Callaway}
\author[T.S.~Doan]{Doan Thai Son}
\author[J.S.W.~Lamb]{Jeroen S.W.~Lamb}
\author[M.~Rasmussen]{Martin Rasmussen \vspace{2ex}\\ Department of Mathematics\\ Imperial College London\\ 180 Queen's Gate\\ London SW7 2AZ\\ United Kingdom}

\date{\today}


\thanks{The fourth author was supported by an EPSRC Career Acceleration Fellowship (2010--2015)}

\maketitle

\begin{abstract}
We develop the \emph{dichotomy spectrum} for random dynamical system and demonstrate its use in the characterization of pitchfork bifurcations for random dynamical systems with additive noise.\\[0.1cm]
  Crauel and Flandoli \cite{Crauel_98_1} had shown earlier that adding noise to a system with a deterministic pitchfork bifurcation yields a unique attracting random fixed point with negative Lyapunov exponent throughout,  thus "destroying" this bifurcation. Indeed, we show that in this example the dynamics before and after the underlying deterministic bifurcation point are topologically equivalent.\\[0.1cm]
  However, in apparent paradox to \cite{Crauel_98_1}, we show that there is after all a qualitative change in the random dynamics at the underlying deterministic bifurcation point, characterized by the transition from a hyperbolic to a non-hyperbolic dichotomy spectrum. This breakdown manifests itself also in the loss of uniform attractivity, a loss of experimental observability of the Lyapunov exponent, and a loss of equivalence under uniformly continuous topological conjugacies.
\end{abstract}


\section{Introduction}
Despite its importance for applications, relatively little progress has been made towards the development of a bifurcation theory for random dynamical systems. Main contributions have been made by Ludwig~Arnold  and co-workers \cite{Arnold_98_1}, distinguishing between \emph{phenomenological} (P-) and a \emph{dynamical} (D-) bifurcations. P-bifurcations refer to qualitative changes in the profile of stationary probability densities \cite{Sri_90_1}. This concept carries substantial drawbacks such as providing reference only to static properties, and not being independent of the choice of coordinates. D-bifurcations refer to the bifurcation of a new invariant measure from a given invariant reference measure, in the sense of weak convergence, and are associated with a qualitative change in the Lyapunov spectrum. They have been studied mainly in the case of multiplicative noise \cite{Baxendale_94_1,Crauel_99_1,Wang_Unpub_1}, and numerically \cite{Arnold_99_2,Keller_99_1}.

In this paper, we contribute to the bifurcation theory of random dynamical systems by shedding new light on the influential paper \emph{Additive noise destroys a pitchfork bifurcation} by Crauel and Flandoli \cite{Crauel_98_1}, in which the stochastic differential equation
\begin{equation}\label{sdeintro}
  \rmd x= \opintb{\alpha x-x^3}\rmd t+ \sigma \rmd W_t\,,
\end{equation}
with two-sided Wiener process $(W_t)_{t\in\R}$ on a probability space $(\Omega,\cF,\mP)$, was studied. In the deterministic (noise-free) case,  $\sigma = 0$, this system has a pitchfork bifurcation of equilibria: if $\alpha<0$ there is one equilibrium ($x=0$) which is globally attractive, and if $\alpha>0$, the trivial equilibrium is repulsive and there are two additional attractive equilibria $\pm \sqrt{\alpha}$. \cite{Crauel_98_1} establish the following facts in the presence of noise, i.e.~when $\sigma>0$:
\begin{itemize}
  \item[(i)]
  For all $\alpha\in\R$, there is a unique globally attracting random fixed point $\set{a_\alpha(\omega)}_{\omega\in\Omega}$.
  \item[(ii)]
 The Lyapunov exponent associated to $\set{a_\alpha(\omega)}_{\omega\in\Omega}$ is negative for all $\alpha\in\R$.
\end{itemize}
As a result,  \cite{Crauel_98_1} concludes that the pitchfork bifurcation is destroyed by the additive noise. (This refers to the absence of  D-bifurcation, as \eqref{sdeintro} admits a qualitative change  P-bifurcation, see \cite[p.~473]{Arnold_98_1}.)
However, we are inclined to argue that the pitchfork bifurcation is not destroyed by additive noise, on the basis of the following additional facts concerning the dynamics near the bifurcation point, that we obtain in this paper:
\begin{itemize}
  \item[(i)]
The attracting random fixed point $\set{a_\alpha(\omega)}_{\omega\in\Omega}$ is uniformly attractive only if $\alpha<0$
(\theoref{Theorem_1}).
\item[(ii)] At the bifurcation point there is a change in the practical observability of the Lyapunov exponent: 
when $\alpha<0$ all finite-time Lyapunov exponents are negative, but when $\alpha>0$ there is a positive probability to observe positive finite-time Lyapunov exponents, irrespectively of the length of time interval under consideration
(\theoref{Finite-timeBifurcation}).
  \item[(ii)]
  The bifurcation point $\alpha=0$ is characterized by a qualitative change in the dichotomy spectrum associated to $\set{a_\alpha(\omega)}_{\omega\in\Omega}$ (\theoref{DichotomySpetrumc-Bifurcation}). In addition, we show that the dichotomy spectrum is directly related to the observability range of the finite-time Lyapunov spectrum (\theoref{theo_1}).
\end{itemize}
In light of these findings, we thus argue for the recognition of qualitative properties of the dichotomy spectrum as an additional indicator for bifurcations of random dynamical systems.
Spectral studies of random dynamical systems have focused mainly on Lyapunov exponents \cite{Arnold_98_1,Cong_97_1}, but here we develop an alternative spectral theory based on exponential dichotomies that is related to the Sacker--Sell (or dichotomy) spectrum for nonautonomous differential equations. The original construction due to R.J.~Sacker and G.R.~Sell \cite{Sacker_78_2} requires a compact base set (which can be obtained, for instance, from an almost periodic differential equation). Alternative approaches to the dichotomy spectrum
\cite{Aulbach_01_2,BenArtzi_93_1,Rasmussen_09_1,Rasmussen_10_2,Siegmund_02_4} hold in the general non-compact case, and we use similar techniques for the construction of the dichotomy spectrum by combining them with ergodic properties of the base flow. We note that the relationship between the dichotomy spectrum and Lyapunov spectrum has also been explored in \cite{Johnson_87_1} in the special case that the base space of a random dynamical system is a compact metric space, but our setup does not require a topological structure of the base.

In analogy to the corresponding bifurcation theory for one-dimensional deterministic dynamical systems, we finally study whether the pitchfork bifurcation with additive noise can be characterized in terms of a breakdown of topologically equivalence. We recall that two random dynamical systems $(\theta,\vp_1)$ and $(\theta,\vp_2)$ are said to be topologically equivalent if there are families $\set{h_\omega}_{\omega\in\Omega}$ of homeomorphisms of the state space such that $\vp_2(t, \omega, h_\omega(x)) = h_{\theta_t\omega}(\vp_1(t, \omega, x))$, almost surely. We establish the following results for the stochastic differential equation \eqref{sdeintro}:
\begin{itemize}
  \item[(i)]
  Throughout the bifurcation, i.e.~for $|\alpha|$ sufficiently small, the resulting dynamics are topologically equivalent (\theoref{Theorem2}).
  \item[(ii)]  
  There does not exist a uniformly continuous topological conjugacy between the dynamics of cases with positive and negative   parameter $\alpha$ (\theoref{theo1}).
\end{itemize}
These results lead us to propose the association of bifurcations of random dynamical systems with a breakdown of \emph{uniform} topological equivalence, rather than the weaker form of general topological equivalence with no requirement on uniform continuity of the involved conjugacy. Note that uniformity of equivalence transformations plays an important role in the notion of equivalence for nonautonomous linear systems (i.e.~in contrast to random systems, the base set of nonautonomous systems is not a probability but a topological space), see \cite{Palmer_79_1}.

%
%

This paper is organised as follows. In \secref{secinvpred}, invariant projectors and exponential dichotomies are introduced for random dynamical systems. \secref{lin.sec1} is devoted to the development of the dichotomy spectrum. In \secref{sec_bif}, we
discuss the pitchfork bifurcation with additive noise, reviewing the results of \cite{Crauel_98_1} and develop our main results in relationship to the dichotomy spectrum. Finally, in \secref{sec_1}, we discuss the existence (and absence) of (uniform) topological equivalence of the dynamics in the neighbourhood of the bifurcation point.
Important preliminaries on random dynamical systems are provided in the appendix.

\section{Exponential dichotomies for random dynamical systems}\label{secinvpred}

In this section, we define invariant projectors and exponential dichotomies as tools to describe hyperbolicity and (un)-stable manifolds of linear random dynamical systems.

Let $(\Omega,\cF,\mP)$ be a probability space and $(X,d)$ be a metric space. A \emph{random dynamical system} $(\theta,\vp)$ (RDS for short) consists of a metric dynamical system $\theta:\T\times \Omega \to \Omega$ (which models the noise, see Appendix) and a $(\cB(\T)\otimes\cF\otimes\cB(X),\cB(X))$-measurable mapping $\vp:\T\times \Omega\times X \to X$ (which models the dynamics of the system) fulfilling
\begin{itemize}
  \item[(i)] $\vp(0,\omega,x)= x$ for all $\omega\in\Omega$ and $x\in X$,
  \item[(ii)] $\vp(t+s,\omega,x) = \vp(t,\theta_s\omega,\vp(s,\omega,x))$ for all $t,s\in \T$, $\omega\in\Omega$ and $x\in X$.
\end{itemize}
Note that we frequently use the abbreviation $\vp(t,\omega)x$ for $\vp(t,\omega,x)$ (even if the random dynamical systems under consideration is nonlinear). We also say that a random dynamical system $(\theta,\vp)$ is ergodic if $\theta$ is ergodic.

For the spectral theory part of this paper, suppose that the phase space $X$ is given by the Euclidean space $\R^d$. A random dynamical system $(\theta, \vp)$ is called \emph{linear} if for given $\alpha,\beta\in \R$, we have
\begin{displaymath}
  \vp(t,\omega)(\alpha x +\beta y)
    =
  \alpha\vp(t,\omega)x+\beta\vp(t,\omega)y
\end{displaymath}
for all $t\in \T$, $\omega\in\Omega$ and $x,y\in\R^d$. Given a linear random dynamical system $(\theta,\vp)$, there exists a corresponding matrix-valued function $\Phi: \T \times \Omega\to \R^{d\times d}$ with $\Phi(t,\omega)x = \vp(t,\omega)x$ for all $t\in\T$, $\omega\in\Omega$ and $x\in \R^d$.

Given a linear random dynamical system $(\theta, \Phi)$, an invariant random set $M$ (see Appendix) is called \emph{linear random set} if for each $\omega\in \R$, the set $M(\omega)$ is a linear subspace of $\R^d$. Given linear random sets $M_1, M_2$,
\begin{displaymath}
  \omega \mapsto M_1(\omega) \cap M_2(\omega)
    \qandq
  \omega \mapsto M_1(\omega) + M_2(\omega)
\end{displaymath}
are also linear random sets, denoted by $M_1\cap M_2$ and $M_1+M_2$, respectively. A finite sum $M_1 + \dots + M_n$ of linear random sets is called \emph{Whitney sum} $M_1 \oplus \dots \oplus M_n$ if $M_1(\omega) \oplus \dots\oplus M_n(\omega) =  \R^d$ holds for almost all $\omega\in\Omega$.

An \emph{invariant projector} of $(\theta,\vp)$ is a measurable function $P: \Omega \to \R^{d \times d}$ with
\begin{displaymath}
  P(\omega)
    =
  P(\omega)^2
    \qandq
  P(\theta_t\omega)\Phi(t,\omega)
    =
  \Phi(t,\omega)P(\omega) \fa t \in\R \mand \omega \in\Omega\,.
\end{displaymath}
The \emph{range}
\begin{displaymath}
  \cR(P) := \setb{(\omega, x) \in \Omega \times \R^d: x \in \cR P(\omega)}
\end{displaymath}
and the \emph{null space}
\begin{displaymath}
  \cN(P) := \setb{(\omega, x ) \in \Omega \times \R^d: x \in \cN P(\omega)}
\end{displaymath}
of an invariant projector $P$ are linear random sets of $(\theta,\vp)$ such that $\cR(P) \oplus \cN(P) = \Omega \times \R^d$.

The following proposition says that, provided ergodicity, the dimensions of the range and the null space of an invariant projector are almost surely constant.

\begin{proposition}\label{Lemma1}
  Let $P:\Omega\to  \R^{d\times d}$ be an invariant projector of an ergodic linear random dynamical system $(\theta,\vp)$. Then
  \begin{itemize}
    \item [(i)]
    the mapping $\omega\mapsto \rk P(\omega)$ is measurable, and
    \item [(ii)]
    $\rk P(\omega)$ is almost surely constant.
  \end{itemize}
\end{proposition}

\begin{proof}
  (i) We first show that the mapping  $A\mapsto \dim A$ on $\R^{d\times d}$ is lower semi-continuous. For this purpose, let $\set{A_k}_{k\in\N}$ be a sequence of matrices in $\R^{d\times d}$ which converges to $A\in\R^{d\times d}$, and define $r:=\dim A$. Then there exist non-zero vectors $x_1,\dots,x_r$ such that $Ax_1,\dots,Ax_r$ are linearly independent, which implies that $\det [Ax_1,\dots,Ax_r,x_{r+1},\dots,x_d]\not=0$ for some vectors $x_{r+1},\dots, x_d\in\R^d$. Since $\lim_{k\to\infty} A_k=A$, one gets
  \begin{displaymath}
    \lim_{k\to\infty} \det [A_k x_1,\dots,A_k x_r,x_{r+1},\dots,x_d]=\det [Ax_1,\dots,Ax_r,x_{r+1},\dots,x_d]\,.
  \end{displaymath}
  Hence, there exists a $k_0\in\N$ such that vectors $A_kx_1,\dots,A_kx_r$ are linearly independent for $k\ge k_0$, and thus, $\dim A_k\geq r$ for all $k\ge k_0$. Consequently, the lower semi-continuity of the mapping $A \mapsto \dim A$ is proved. Therefore, the map $\R^{d\times d}\to  \N, A\mapsto \dim A$ is the limit of a monotonically increasing sequence of continuous functions \cite{Tong_52_1} and thus is measurable. The proof of this part is complete.
  (ii) By invariance of $P$, we get that
  \begin{displaymath}
  P(\theta_t\omega)=\Phi(t,\omega)P(\omega)\Phi(t,\omega)^{-1},
  \end{displaymath}
  which implies that $\dim P(\theta_t\omega)=\dim P(\omega)$. This together with ergodicity of $\theta$ and measurability of the map $\omega\mapsto \dim P(\omega)$ as shown in (i) gives that $\dim P(\omega)$ is almost constant.
\end{proof}
According to \propref{Lemma1}, the rank of an invariant projector $P$ can be defined via
\begin{displaymath}
  \rk P := \dim \cR(P) := \dim \cR P(\omega)\faa \omega\in \Omega\,,
\end{displaymath}
and one sets
\begin{displaymath}
  \dim \cN(P) := \dim \cN P(\omega) \faa \omega\in \Omega.
\end{displaymath}

The following notion of an exponential dichotomy describes uniform exponential splitting of linear random dynamical systems.

\begin{definition}[Exponential dichotomy]\label{ED}
  Let $(\theta, \Phi)$ be a linear random dynamical system, and let $\gamma\in\R$ and $P_\gamma: \Omega \to \R^{d \times d}$ be an invariant projector of $(\theta,\vp)$. Then $(\theta,\vp)$ is said to admit an \emph{exponential dichotomy} with growth rate $\gamma\in\R$, constants
  $\alpha>0$, $K\ge 1$ and projector $P_\gamma$ if for almost all $\omega\in\Omega$, one has
  \begin{align*}
    \norm{\Phi(t, \omega)P_\gamma(\omega)}
      &\le
    K e^{(\gamma-\alpha)t}  \fa t \ge 0\,,\\
    \norm{\Phi(t,\omega)(\1-P_\gamma(\omega))}
      &\le
    K e^{(\gamma+\alpha)t} \fa t\le 0\,.
  \end{align*}
\end{definition}

The following proposition shows that the ranges and null spaces of invariant projectors are given by sums of Oseledets subspaces.

\begin{proposition}
  Let $(\theta,\Phi)$ be an ergodic linear random dynamical system which satisfies the integrability condition of Oseledets Multiplicative Ergodic Theorem (see Appendix). Let $\lambda_1>\dots>\lambda_p$ and $O_1(\omega),\dots, O_p(\omega)$ denote the Lyapunov exponents and the associated Oseledets subspaces of $(\theta,\Phi)$, respectively, and suppose that $\Phi$ admits an exponential dichotomy with growth rate $\gamma\in\R$ and projector $P$. Then the following statements hold:
  \begin{itemize}
    \item [(i)]
    $\gamma\not\in\{\lambda_1,\dots,\lambda_p\}$.
    \item [(ii)]
    Define $k:=\max\setb{i\in\{0,\dots,p\}: \lambda_i>\gamma}$ with the convention that $\lambda_0=\infty$. Then for almost all $\omega\in\Omega$, one has
    \begin{displaymath}
      \cN P(\omega)=\bigoplus_{i=1}^{k}O_i(\omega) \qandq \cR P(\omega)=\bigoplus_{i=k+1}^{p}O_i(\omega)\,.
    \end{displaymath}
  \end{itemize}
\end{proposition}

\begin{proof}
  (i) Suppose to the contrary that $\gamma=\lambda_k$ for some $k\in\{1,\dots,p\}$. Because of the Multiplicative Ergodic Theorem, we have
  \begin{equation}\label{eqn1}
    \lim_{t\to\infty}\frac{1}{t}\ln\|\Phi(t,\omega)v\|=\lambda_k=\gamma\quad \fa v\in O_k(\omega)\setminus\{0\}\,.
  \end{equation}
  On the other hand, for all $v\in \cR P_\gamma(\omega)$ we get $\|\Phi(t,\omega)v\|\leq Ke^{(\gamma-\alpha)t}\|v\|$ for all $t\geq 0$. Thus,
  \begin{displaymath}
    \limsup_{t\to\infty}\frac{1}{t}\ln\|\Phi(t,\omega)v\|\leq \gamma-\alpha \fa  v\in \cR P(\omega)\,,
  \end{displaymath}
  which together with \eqref{eqn1} implies that $O_k(\omega)\cap \cR P(\omega)=\{0\}$. Similarly, using the fact that
  \begin{equation*}
    \lim_{t\to-\infty}\frac{1}{t}\ln\|\Phi(t,\omega)v\|=\lambda_k=\gamma \fa v\in O_k(\omega)\setminus\{0\}
  \end{equation*}
  and \defref{ED}, we obtain that $O_k(\omega)\cap \cN P(\omega)=\{0\}$. Consequently, $O_k(\omega)=\{0\}$ and it leads to a contradiction.

  (ii) Let $v\in \cR P(\omega)\setminus\{0\}$ be arbitrary. Then, according to Definition \ref{ED} and the definition of $k$ we obtain that
  \begin{equation}\label{Eq9}
  \lim_{t\to\infty}\frac{1}{t}\ln\|\Phi(t,\omega)v\|\leq \gamma-\alpha<\lambda_k.
  \end{equation}
  Now we write $v$ in the form $v=v_i+v_{i+1}+\dots+v_p$, where $i \in\set{1,\dots,p}$ with $v_i\not=0$ and $v_j\in O_j(\omega)$ for all $j=i,\dots,p$. Using the fact that for $j\in\set{i,\dots,p}$ with $v_j\not=0$
  \begin{displaymath}
  \lim_{t\to\infty}\frac{1}{t}\ln\|\Phi(t,\omega)v_j\|=\lambda_j\leq \lambda_i,
  \end{displaymath}
  we obtain that
  \begin{displaymath}
  \lim_{t\to\infty}\frac{1}{t}\ln\|\Phi(t,\omega)v\|=\lambda_i,
  \end{displaymath}
  which together with \eqref{Eq9} implies that $i\geq k+1$ and therefore $\cR P(\omega)\subset\bigoplus_{i=k+1}^{p}O_i(\omega)$. Similarly, we also get that $\cN P(\omega)\subset \bigoplus_{i=1}^{k}O_i(\omega)$. On the other hand,
  \begin{displaymath}
  \R^d=\cN P(\omega)\oplus \cR P(\omega)=\bigoplus_{i=1}^{k}O_i(\omega)\oplus \bigoplus_{i=k+1}^{p}O_i(\omega).
  \end{displaymath}
  Consequently, we have $\cR P(\omega)=\bigoplus_{i=k+1}^{p}O_i(\omega)$ and $\cN P(\omega)\subset \bigoplus_{i=1}^{k}O_i(\omega)$. The proof is complete.
\end{proof}

The monotonicity of the exponential function implies the following basic criteria for the existence of exponential dichotomies.

\begin{lemma}\label{lin.lemma4}
  Suppose that the linear random system $(\theta,\Phi)$ admits an exponential dichotomy with growth rate $\gamma$ and projector  $P_\gamma$. Then the following statements are fulfilled:
  \begin{itemize}
    \item[(i)]
    If $P_\gamma \equiv \1$ almost surely, then $(\theta,\Phi)$ admits an exponential dichotomy with growth rate $\zeta$ and invariant projector $P_\zeta \equiv \1$ for all $\zeta>\gamma$.
    \item[(ii)]
    If $P_\gamma \equiv 0$ almost surely, then $(\theta,\Phi)$ admits an exponential dichotomy with growth rate $\zeta$ and invariant projector $P_\zeta \equiv 0$ for all $\zeta<\gamma$.
  \end{itemize}
\end{lemma}

Given $\gamma\in \R$, a function $g: \R \to \R^d$ is called \emph{$\gamma^+$-exponentially bounded} if $\sup_{t \in \R\cap[0,\infty)} \norm{g(t)}e^{-\gamma t} < \infty$. Accordingly, one says that a function $g: \R \to \R^d$ is \emph{$\gamma^-$-exponentially bounded} if $\sup_{t \in \R\cap(-\infty,0]}\,\norm{g(t)}e^{-\gamma t} < \infty$.

We define for all $\gamma \in\R$
\begin{displaymath}
  \cS^\gamma := \setb{(\omega, x) \in \Omega \times \R^d: \Phi(\cdot, \omega)x \text{ is } \gamma^+\text{-exponentially bounded}} \,,
\end{displaymath}
and
\begin{displaymath}
  \cU^\gamma := \setb{(\omega, x) \in \Omega \times \R^d: \Phi(\cdot, \omega)x \text{ is } \gamma^-\text{-exponentially bounded}} \,.
\end{displaymath}
It is obvious that $\cS^\gamma$ and $\cU^\gamma$ are linear invariant random sets of $(\theta,\vp)$, and given $\gamma \le \zeta$, the relations $\cS^\gamma \subset \cS^\zeta$ and $\cU^\gamma \supset \cU^\zeta$ are fulfilled.

The relationship between the projectors of exponential dichotomies with growth rate $\gamma$ and the sets $\cS^\gamma$ and $\cU^\gamma$ will now be discussed.

\begin{proposition}\label{lin.prop2}
  If the linear random dynamical system $(\theta,\Phi)$ admits an exponential dichotomy with growth rate $\gamma$ and projector $P_\gamma$, then $\cN(P_\gamma) = \cU^\gamma$ and $\cR(P_\gamma) = \cS^\gamma$ almost surely.
\end{proposition}

\begin{proof}
  Suppose that $(\theta,\Phi)$ admits an exponential dichotomy with growth rate $\gamma$, constants $\alpha$, $K$ and projector  $P_\gamma$. This means that for almost all $\omega\in\Omega$, one has
  \begin{eqnarray*}
    \norm{\Phi(t,\omega)P_\gamma(\omega)}
      \minequ{\le}
    Ke^{(\gamma-\alpha) t} \fa t\ge 0\,,\\
    \norm{\Phi(t,\omega)(\1-P_\gamma(\omega))}
      \minequ{\le}
    K e^{(\gamma+\alpha)t} \fa t\le 0\,.
  \end{eqnarray*}
  We now prove the relation $\cN(P_\gamma) = \cU^\gamma$.
  \\
  ($\supseteq$) Choose $(\omega, x) \in \cU^\gamma$ arbitrarily. This implies $\norm{\Phi(t,\omega)x} \le Ce^{\gamma t}$ for all $t \le 0$ with some real constant $C>0$. Write $x = x_1+x_2$ with $x_1 \in \cR P_\gamma(\omega)$ and $x_2 \in \cN P_\gamma(\omega)$. Hence, for all $t\le 0$,
  \begin{align*}
    \norm{x_1}
      &=
    \norm{\Phi(-t, \theta_t\omega)\Phi(t,\omega)P_\gamma(\omega)x}
      =
    \norm{\Phi(-t, \theta_t\omega)P_\gamma(\theta_t\omega)\Phi(t,\omega)x}\\
      &\le
    K e^{-(\gamma-\alpha)t}\norm{\Phi(t,\omega)x}
      \le
    C K e^{-(\gamma-\alpha)t}e^{\gamma t} =C K e^{\alpha t}\,.
  \end{align*}
  The right hand side of this inequality converges to zero in the limit $t \to -\infty$. This implies $x_1=0$, and thus,   $(\omega,x)\in\cN(P_\gamma)$.
  \\
  ($\subseteq$) Choose $(\omega, x) \in \cN(P_\gamma)$. Thus, for all $t \le 0$, the relation $\norm{\Phi(t,\omega)x} \le K
  e^{(\gamma+\alpha)t}\norm{x}$ is fulfilled. This means that $\Phi(\cdot,\omega)x$ is $\gamma^-$-exponentially bounded.
  \\
  The proof of statement concerning the range of the projector is treated analogously.
\end{proof}

\section{The dichotomy spectrum for random dynamical systems}\label{lin.sec1}

We introduce the dichotomy spectrum for random dynamical systems in this section. For the definition of the dichotomy spectra, it is crucial for which growth rates, a linear random dynamical system $(\theta,\Phi)$ admits an exponential dichotomy. The growth rates $\gamma = \pm \infty$ are not excluded from our considerations;  in particular, one says that $(\theta,\Phi)$ admits an exponential dichotomy with growth rate $\infty$ if there exists a $\gamma \in \R$ such that $(\theta,\Phi)$ admits an exponential dichotomy with growth rate $\gamma$ and projector $P_\gamma \equiv \1$. Accordingly, one says that $(\theta,\Phi)$ admits an exponential dichotomy with growth rate $-\infty$ if there exists a $\gamma \in \R$ such that $(\theta,\Phi)$ admits an exponential dichotomy with growth rate $\gamma$ and projector $P_\gamma \equiv 0$.

\begin{definition}[Dichotomy spectrum]\label{def_lin_3}
  Consider the linear random dynamical system $(\theta,\Phi)$. Then the \emph{dichotomy spectrum} $(\theta,\Phi)$ is defined by
  \begin{displaymath}
    \Sigma
      :=
    \setb{\gamma \in \Rq: (\theta,\Phi)\text{ does not admit an exponential dichotomy with growth rate }\gamma}\,.
  \end{displaymath}
  The corresponding \emph{resolvent sets} is defined by $\rho := \Rq \setminus \Sigma$.
\end{definition}

The aim of the following lemma is to analyze the topological structure of the resolvent sets.

\begin{lemma}\label{lin.lemma2}
  Consider the resolvent set $\rho$ of a linear random dynamical system $(\theta,\Phi)$. Then $\rho \cap \R$ is open. More precisely, for all $\gamma \in \rho\cap \R$, there exists an $\eps >0$ such that $B_\eps(\gamma) \subset \rho$. Furthermore, the relation $\rk P_\zeta = \rk P_\gamma$ is (almost surely) fulfilled for all $\zeta \in B_\eps(\gamma)$ and every invariant projector $P_\gamma$ and $P_\zeta$ of the exponential dichotomies of $(\theta,\Phi)$ with growth rates $\gamma$ and $\zeta$, respectively.
\end{lemma}

\begin{proof}
  Choose $\gamma \in \rho$ arbitrarily. Since $(\theta,\Phi)$ admits an exponential dichotomy with growth rate $\gamma$, there exist an invariant projector $P_\gamma$ and constants $\alpha>0$, $K\ge 1$ such that for almost all $\omega\in\Omega$, one has
  \begin{alignat*}{2}
    \norm{\Phi(t,\omega)P_\gamma(\omega)} &\le Ke^{(\gamma-\alpha) t} &\fa t\ge 0\,,\\
    \norm{\Phi(t,\omega)(\1-P_\gamma(\omega))} &\le K e^{(\gamma+\alpha)t} &\fa t\le 0 \,.
  \end{alignat*}
  Set $\eps := \frac{1}{2}\alpha$, and choose $\zeta \in B_\eps(\gamma)$. It follows that for almost all $\omega\in\Omega$,
  \begin{alignat*}{2}
    \norm{\Phi(t,\omega)P_\gamma(\omega)} &\le  Ke^{(\zeta-\frac{\alpha}{2})t}&\fa& t\ge 0 \,,\\
    \norm{\Phi(t,\omega)(\1-P_\gamma(\omega))} &\le K e^{(\zeta+\frac{\alpha}{2})t} &\fa& t\le 0 \,.
  \end{alignat*}
  This yields $\zeta \in \rho$, and it follows that $\rk P_\zeta = \rk P_\gamma$ for any projector $P_\zeta$ of the exponential dichotomy with growth rate $\zeta$. This finishes the proof of this lemma.
\end{proof}

\begin{lemma}\label{lin.lemma3}
  Consider the resolvent set $\rho$ of a linear random dynamical system $(\theta,\Phi)$, and let $\gamma_1,\gamma_2 \in \rho \cap \R$ such that $\gamma_1<\gamma_2$. Moreover, choose invariant projectors $P_{\gamma_1}$ and $P_{\gamma_2}$ for the corresponding exponential dichotomies with growth rates $\gamma_1$ and $\gamma_2$. Then the relation $\rk P_{\gamma_1} \le \rk P_{\gamma_2}$ holds. In addition, $[\gamma_1,\gamma_2] \subset \rho$ is fulfilled if and only if $\rk P_{\gamma_1} = \rk P_{\gamma_2}$.
\end{lemma}

\begin{proof}
  The relation $\rk P_{\gamma_1} \le \rk P_{\gamma_2}$ is a direct consequence of \propref{lin.prop2}, since $\cS^{\gamma_1} \subset \cS^{\gamma_2}$ and $\cU^{\gamma_1} \supset \cU^{\gamma_2}$. Assume now that $[\gamma_1,\gamma_2] \subset \rho$. Arguing contrapositively, suppose that $\rk P_{\gamma_1} \not= \rk P_{\gamma_2}$, and choose invariant projectors $P_\gamma$, $\gamma \in (\gamma_1,\gamma_2)$, for the exponential dichotomies of $(\theta,\Phi)$ with growth rate $\gamma$. Define
  \begin{displaymath}
    \zeta_0
    :=
    \sup\,\setb{\zeta \in [\gamma_1,\gamma_2]: \rk P_\zeta \not= \rk P_{\gamma_2}}\,.
  \end{displaymath}
  Due to \lemmaref{lin.lemma2}, there exists an $\eps >0$ such that $\rk P_{\zeta_0} = \rk P_\zeta$ for all $\zeta \in B_\eps(\zeta_0)$. This is a contradiction to the definition of $\zeta_0$. Conversely, let $\rk P_{\gamma_1} = \rk P_{\gamma_2}$. Because of $\rk P_{\gamma_1} = \rk P_{\gamma_2}$, \propref{lin.prop2} yields $\cN(P_{\gamma_1}) = \cN(P_{\gamma_2})$ almost surely, and $P_{\gamma_2}$ is an invariant projector of the exponential dichotomy with growth rate $\gamma_1$.  Thus, one obtains for almost all $\omega\in\Omega$,
  \begin{displaymath}
    \norm{\Phi(t,\omega) P_{\gamma_2}(\omega)} \le K_1 e^{(\gamma_1 - \alpha_1)t} \fa t\ge 0
  \end{displaymath}
  for some $K_1\ge 1$ and $\alpha_1 >0$. $P_{\gamma_2}$ is also projector of the exponential dichotomy on $\R_0^-$ with growth rate $\gamma_2$. Hence, for almost all $\omega\in\Omega$, one gets
  \begin{displaymath}
    \normb{\Phi(t,\omega) (\1-P_{\gamma_2}(\omega))} \le K_2 e^{(\gamma_2 + \alpha_2)t} \fa t\le0
  \end{displaymath}
  with some $K_2\ge 1$ and $\alpha_2 >0$. For all $\gamma \in [\gamma_1, \gamma_2]$, these two inequalities imply by setting $K:= \max\,\set{K_1, K_2}$ and $\alpha:= \min \,\set{\alpha_1, \alpha_2}$ that for almost all $\omega\in\Omega$,
  \begin{alignat*}{2}
    \norm{\Phi(t,\omega) P_{\gamma_2}(\omega)} &\le  K e^{(\gamma - \alpha)t} &\fa& t\ge0\,,\\
    \norm{\Phi(t,\omega)(\1-P_{\gamma_2}(\omega))} &\le K e^{(\gamma + \alpha)t} &\fa&  t\le0\,.
  \end{alignat*}
  This means that $\gamma \in \rho$, and thus, $[\gamma_1,\gamma_2] \subset \rho$.
\end{proof}

For an arbitrarily chosen $a \in \R$, define
\begin{displaymath}
  [-\infty, a]:=(-\infty,a] \cup \set{-\infty}\,,\quad\quad\quad   [a, \infty]  := [a, \infty) \cup \set{\infty}
\end{displaymath}
and
\begin{displaymath}
[-\infty,-\infty]:= \set{-\infty}, \quad\quad\  [\infty,\infty]:= \set{\infty}, \quad\quad \quad [-\infty,\infty] := \Rq\,.
\end{displaymath}

The following \emph{Spectral Theorem}, describes that the dichotomy spectrum consists of at least one and at most $d$ closed intervals.

\begin{theorem}[Spectral Theorem]\label{lin.theo1}
  Let $(\theta, \Phi)$ be a linear random dynamical system with dichotomy spectrum $\Sigma$. Then there exists an $n \in \set{1,\dots,d}$ such that
  \begin{displaymath}
    \Sigma = [a_1, b_1] \cup \dots \cup [a_n,b_n]
  \end{displaymath}
  with $-\infty \le a_1 \le b_1< a_2 \le b_2 <\dots<a_n\le b_n\le \infty$.
\end{theorem}

\begin{proof}
  Due to \lemmaref{lin.lemma2}, the resolvent set $\rho \cap \R$ is open. Thus, $\Sigma \cap \R$ is the disjoint union of closed intervals. The relation $(-\infty, b_1] \subset \Sigma$ implies $[-\infty, b_1] \subset \Sigma$, because the assumption of the existence of a $\gamma \in \R$ such that $(\theta,\vp)$ admits an exponential dichotomy with growth rate $\gamma$ and projector $P_\gamma \equiv 0$ leads to $(-\infty,\gamma] \subset \rho$ using \lemmaref{lin.lemma4}, and this is a contradiction. Analogously, it follows from $[a_n, \infty) \subset \Sigma$ that $[a_n, \infty] \subset \Sigma$. To show the relation $n \le d$,  assume to the contrary that $n \ge d+1$. Thus, there exist
  \begin{displaymath}
    \zeta_1<\zeta_2<\dots<\zeta_d \in \rho
  \end{displaymath}
  such that the $d+1$ intervals $(-\infty, \zeta_1)\,,\, (\zeta_1, \zeta_2)\,,\, \dots,\, (\zeta_d, \infty)$ have nonempty intersection with the spectrum $\Sigma$. It follows from Lemma \ref{lin.lemma3} that
  \begin{displaymath}
    0 \le \rk P_{\zeta_1} < \rk P_{\zeta_2} < \dots < \rk P_{\zeta_d} \le d
  \end{displaymath}
  is fulfilled for invariant projectors $P_{\zeta_i}$ of the exponential dichotomy with growth rate $\zeta_i$, $i\in \set{1,\dots,n}$. This implies either $\rk P_{\zeta_1} =0$ or $\rk P_{\zeta_d} =d$. Thus, either
  \begin{displaymath}
    [-\infty, \zeta_1] \cap \Sigma = \emptyset \qorq [\zeta_d, \infty] \cap \Sigma = \emptyset
  \end{displaymath}
  is fulfilled, and this is a contradiction. To show $n\ge 1$, assume that $\Sigma= \emptyset$. This implies $\set{-\infty, \infty} \subset \rho$. Thus, there exist $\zeta_1, \zeta_2 \in \R$ such that $(\theta,\vp)$ admits an exponential dichotomy with growth rate $\zeta_1$ and projector $P_{\zeta_1} \equiv 0$ and an exponential dichotomy with growth rate $\zeta_2$ and projector $P_{\zeta_2} \equiv \1$. Applying \lemmaref{lin.lemma3}, one gets $(\zeta_1, \zeta_2) \cap \Sigma \not= \emptyset$. This contradiction yields $n\ge 1$ and finishes the proof of the theorem.
\end{proof}

Each spectral interval is associated to a so-called spectral manifold, which generalises the stable and unstable manifolds obtained by the ranges and null spaces of invariant projectors of exponential dichotomies.

\begin{theorem}[Spectral manifolds]\label{lin.theo6}
  Consider the dichotomy spectrum
  \begin{displaymath}
    \Sigma = [a_1, b_1] \cup\dots\cup [a_n,b_n]
  \end{displaymath}
  of the linear random dynamical system $(\theta, \Phi)$ and define the invariant projectors $P_{\gamma_0} := 0$, $P_{\gamma_n}:=  \1$, and for $i\in\set{1,\dots,n-1}$, choose $\gamma_i \in (b_i,a_{i+1})$ and projectors $P_{\gamma_i}$ of the nonhyperbolic exponential dichotomy of $(\theta,\vp)$ with growth rate $\gamma_i$. Then the sets
  \begin{displaymath}
    \cW_i := \cR(P_{\gamma_i}) \cap \cN(P_{\gamma_{i-1}}) \fa i \in \set{1,\dots,n}
  \end{displaymath}
  are fiber-wise linear subset of $\R^d$, the so-called \emph{spectral manifolds}, such that
  \begin{displaymath}
    \cW_1 \oplus \dots \oplus \cW_n = \Omega \times \R^d
  \end{displaymath}
  and $\cW_i \not= \Omega \times \set{0}$ for $i \in \set{1,\dots,n}$.
\end{theorem}

\begin{proof}
  The sets $\cW_1,\dots,\cW_n$ obviously have linear fibers. Suppose that there exists an $i \in \set{1,\dots,n}$ with $\cW_i = \R \times   \set{0}$. In case $i=1$ or $i=n$, \lemmaref{lin.lemma4} implies $[-\infty, \gamma_1] \cap \Sigma = \emptyset$ or $[\gamma_{n-1},\infty] \cap \Sigma = \emptyset$, and this is a contradiction. In case $1<i<n$, due to \lemmaref{lin.lemma3}, one obtains
  \begin{displaymath}
    \dim \cW_i
      =
    \dim \opintb{\cR(P_{\gamma_i}) \cap \cN(P_{\gamma_{i-1}})}
      =
    \rk P_{\gamma_i} + d - \rk P_{\gamma_{i-1}} - \dim \opintb{\cR(P_{\gamma_i}) +\cN(P_{\gamma_{i-1}})} \ge 1\,,
  \end{displaymath}
  and this is also a contradiction. Now the relation $\cW_1 \oplus \dots \oplus \cW_n = \Omega \times \R^d$ will be proved. For $1 \le i < j \le n$, due to \ref{lin.prop2}, the relations $\cW_i \subset \cR(P_{\gamma_i})$ and $\cW_j \subset \cN(P_{\gamma_{j-1}}) \subset \cN(P_{\gamma_i})$ are fulfilled. This yields
  \begin{displaymath}
    \cW_i \cap \cW_j \subset \cR(P_{\gamma_i}) \cap \cN(P_{\gamma_i}) = \R\times \set{0}\,,
  \end{displaymath}
  and one obtains
  \begin{eqnarray*}
    \Omega\times \R^d
      \minequ{=}
    \cW_1 + \cN(P_{\gamma_1}) = \cW_1 + \cN(P_{\gamma_1}) \cap \opintb{\cR(P_{\gamma_2}) + \cN(P_{\gamma_2})}\\
      \minequ{=}
    \cW_1 + \cN(P_{\gamma_1}) \cap \cR(P_{\gamma_2}) + \cN(P_{\gamma_2}) = \cW_1 + \cW_2 + \cN(P_{\gamma_2})\,.
  \end{eqnarray*}
  Here, the fact that linear subspaces $E,F,G\subset \R^d$ with $E\supset G$ fulfill $E\cap(F+G)= (E\cap F) + G$ was used. It follows inductively that
  \begin{displaymath}
    \Omega\times\R^d = \cW_1 + \dots+ \cW_n + \cN(P_{\gamma_n}) = \cW_1 + \dots+ \cW_n\,.
  \end{displaymath}
  This finishes the proof of this theorem.
\end{proof}

\begin{remark}
  If the linear random dynamical system $(\theta,\Phi)$ under consideration fulfills the conditions of the Multiplicative Ergodic Theorem, then \propref{Lemma1} implies that the spectral manifolds $\cW_i$ of the above theorem are given by Whitney sums of Oseledets subspaces.
\end{remark}

The remaining part of this section on the dichotomy spectrum will be devoted to study boundedness properties of the spectrum. Firstly, a criterion for boundedness from above and below is provided by the following proposition.

\begin{proposition}
  Consider a linear random dynamical system $(\theta, \Phi)$, let $\Sigma$ denote the dichotomy spectrum of $(\theta, \Phi)$, and define
  \begin{displaymath}
    \alpha^{\pm}(\omega)
    :=
    \left\{
      \begin{array}{c@{\;:\;}l}
        \ln^+\opintb{\norm{\Phi(1,\omega)^{\pm 1}}} & \T=\Z\,,\\
        \ln^+\opintb{\sup_{t\in[0, 1]}\norm{\Phi(t,\omega)^{\pm 1}}} & \T=\R\,.
      \end{array}
    \right.
  \end{displaymath}
  Then $\Sigma$ is bounded from above if and only if
  \begin{displaymath}
    \esssup_{\omega\in\Omega}\alpha^+(\omega)<\infty \,,
  \end{displaymath}
  and $\Sigma$ is bounded from below if and only if
  \begin{displaymath}
    \esssup_{\omega\in\Omega}\alpha^-(\omega)<\infty\,.
  \end{displaymath}
  Consequently, if the dichotomy spectrum $\Sigma$ is bounded, then $\Phi$ satisfies the integrability condition of the Multiplicative Ergodic Theorem.
\end{proposition}

\begin{proof}
  Suppose that $\Sigma$ is bounded from above. Then there exist $K>0$ and $\alpha\in\R$ such that
  \begin{displaymath}
    \|\Phi(t,\omega)\|\leq K e^{\alpha t} \faa \omega\in\Omega\,,
  \end{displaymath}
  which implies that $\esssup_{\omega\in\Omega}\alpha^+(\omega)\leq K e^{|\alpha|}$. On the other hand, suppose that $\esssup_{\omega\in\Omega}\alpha(\omega)<\infty$. Then there exists a measurable set $U$ of probability $1$ such that for all $\omega\in\Omega$ we have $\alpha^+(\omega)\leq e^{\alpha}$ for some positive number $\alpha$. Define
  \begin{displaymath}
    \widetilde \Omega:=\bigcap_{n\in\Z}\theta^n U\,.
  \end{displaymath}
  Due to the measure preserving property of $\theta$, we get that $\mP\opintb{\widetilde \Omega}=1$. Let $\gamma>\alpha$ be arbitrary. Then for all $\omega\in\widetilde \Omega$, we have
  \begin{displaymath}
    \|\Phi(t,\omega)\|\leq e^{\alpha t+\alpha} \fa t > 0 \,,
  \end{displaymath}
  which implies that $\gamma\in \R\setminus \Sigma$. Hence, $\Sigma\subset (-\infty,\alpha]$. Similarly, we get that $\Sigma$ is bounded from below if and only if $\esssup_{\omega\in\Omega}\alpha^-(\omega)<\infty$. This finishes the proof of this proposition.
\end{proof}

The following example shows that there exist linear random dynamical systems which satisfy the integrability condition of the Multiplicative Ergodic Theorem, but which have no bounded dichotomy spectrum.

\begin{example}\label{Example1}
  Let $(\Omega,\mathcal F,\mathbb P)$ be a probability space and $\theta:\R\times\Omega\to \Omega$ be a metric dynamical system which is ergodic and non-atomic. Then there exists, by using \cite[Lemma 2, p.~71]{Halmos_60_1}, a measurable set $U$ of the form
  \begin{equation}\label{Eq10}
    U=\bigcup_{k=1}^\infty\bigcup_{j=0}^{k}\theta_j U_k,
  \end{equation}
  where $U_i$, $i\in\N$, are measurable sets such that
  \begin{itemize}
    \item [(i)]
    for all $k,\ell\in\N$, $i\in\set{0,\dots, k}$ and $j\in\set{0,\dots,\ell}$, we have
    \begin{displaymath}
      \theta_j U_k\cap \theta_i U_{\ell}=\emptyset \quad \mbox{whenever }\, k\not=\ell \mor i\not=j\,,
    \end{displaymath}
    \item [(ii)]
    $0<\mathbb P(U_k)\leq \frac{1}{k^3}$ for all $k\in\N$\,.
  \end{itemize}
  We now define a random variable $a:\Omega\to \R$ by
  \begin{displaymath}
  a(\omega):=
  \left\{
    \begin{array}{c@{\;:\;}l}
      1 & \omega\in \Omega\setminus U\,,  \\
      k & k \mbox{ is even and } \omega\in\theta_j U_k\,, \\
      \frac{1}{k} & k \mbox{ is odd and } \omega\in\theta_j U_k\,.
    \end{array}
  \right.
  \end{displaymath}
  Using the random variable $a$, we define a discrete-time scalar linear random dynamical system $\Phi:\Z\times \Omega \to \R$ by
  \begin{displaymath}
    \Phi(t,\omega)=\left\{
     \begin{array}{c@{\;:\;}l}
       a(\theta_{t-1}\omega) \cdots a(\omega) & t\geq 1\,, \\
       1 & t=0\,,  \\
       a(\theta_{-1}\omega)^{-1}\cdots a(\theta_{t}\omega)^{-1} & t\leq -1\,.
     \end{array}
   \right.
  \end{displaymath}
  A direct computation yields that
  \begin{displaymath}
    \mathbb E \ln^+(\|\Phi(1,\omega)\|)
    =
    \sum_{k=1}^\infty (2k+1)\mP(U_{2k})\ln (2k)
    \leq
    \sum_{k=1}^\infty (2k+1)\frac{\ln (2k)}{8k^3}<\infty\,,
  \end{displaymath}
  and
  \begin{align*}
  \mathbb E \ln^+(\|\Phi(1,\omega)^{-1}\|)
  &=
  \sum_{k=0}^\infty (2k+2)\mP(U_{2k+1})\ln (2k+1)\\
  &\leq
  \sum_{k=1}^\infty (2k+2)\frac{\ln (2k+1)}{(2k+1)^3}<\infty\,.
  \end{align*}
  Then the linear system $\Phi$ satisfies the integrability condition of the Multiplicative Ergodic Theorem. The fact that the dichotomy spectrum of $\Phi$ is unbounded from above follows from
  \begin{displaymath}
    \|\Phi(n,\omega)\|\geq k^{n}\fa \omega\in U_k \mwith k \mbox{ even and } 0\le n\le k\,.
  \end{displaymath}
  Similarly, one can prove that the spectrum is unbounded from below.
\end{example}

\section{Random pitchfork bifurcation}\label{sec_bif}

We first review in \subsecref{subsec1} the main results of \cite{Crauel_98_1}, which concern the one-dimensional stochastic differential equation
\begin{equation}\label{Sde}
  \rmd x= \opintb{\alpha x-x^3}\rmd t+ \sigma \rmd W_t\,,
\end{equation}
depending on real parameters $\alpha$ and $\sigma$ and driven by a two-sided Wiener process $(W_t)_{t\in\R}$. This stochastic differential equation has a unique random fixed point $\set{a_\alpha(\omega)}_{\omega\in\Omega}$ for all $\alpha\in\R$. We then show in \subsecref{subsec2} that there is a qualitative change in the random dynamics at the bifurcation point $\alpha=0$ in the sense that after the bifurcation, the attracting random fixed points $\set{a_\alpha(\omega)}_{\omega\in\Omega}$ have qualitatively different properties for $\alpha<0$ and $\alpha\ge 0$ with respect to uniform attraction, which is lost at the bifurcation point. We also associate this bifurcation in \subsecref{subsec3} with non-hyperbolicity of the spectrum of the linearization at the bifurcation point.

\subsection{Existence of a unique random attracting fixed point}\label{subsec1}

Consider the stochastic differential equation \eqref{Sde}. We first look at the deterministic case $\sigma = 0$. Then for $\alpha<0$, the ordinary differential equation \eqref{Sde} has one equilibrium ($x=0$) which is globally attractive. For positive $\alpha$, the trivial equilibrium becomes repulsive, and there are two additional equilibria, given by $\pm \sqrt{\alpha}$, which are attractive. This also means that the global attractor $K_\alpha$ of the deterministic equation undergoes a bifurcation from a trivial to a nontrivial object. It is given by
\begin{displaymath}
  K_\alpha:=
  \left\{
    \begin{array}{c@{\;:\;}l}
      \set{0} & \alpha \le 0\,,\\%
      \clintb{-\sqrt\alpha,\sqrt\alpha}   & \alpha> 0\,.
    \end{array}
  \right.
\end{displaymath}

It was shown in \cite{Crauel_98_1} that such an attractor bifurcation does not persist for random attractors of the randomly perturbed system where $\sigma >0$, and we will explain the details now.

Firstly, the stochastic differential equation \eqref{Sde} generates a random dynamical system $(\theta:\R\times \Omega\to\Omega,\vp:\R\times \Omega \times \R\to \R)$ which induces a Markov semigroup with transition probabilities $T(x, B)$ for $x\in \R$ and $B\in\cB(\R)$. A probability measure $\rho$ on $\cB(X)$ is called a \emph{stationary measure} for the Markov semigroup if
\begin{displaymath}
  \rho(B) = \int_\R T(x,B) \,\rmd \rho(x) \fa B\in\cB(\R)\,.
\end{displaymath}
It can be shown \cite[p.~474]{Arnold_98_1} that for any $\alpha,\sigma\in\R$, the Markov semigroup associated with \eqref{Sde} admits a unique stationary measure $\rho_{\alpha,\sigma}$ with density
\begin{equation}\label{Eq3}
  p_{\alpha,\sigma}(x)=N_{\alpha,\sigma}\exp\opintb{\tfrac{1}{\sigma^2}(\alpha x^2-\tfrac{1}{2}{x^4})}\,,
\end{equation}
where $N_{\alpha,\sigma}$ is a normalization constant. This stationary measure corresponds to an invariant measure $\mu$ of the random dynamical system $(\theta,\vp)$ generated by \eqref{Sde}. $\mu$ has the disintegration given by
\begin{displaymath}
  \mu_\omega=\lim_{t\to\infty}\vp(t,\theta_{-t}\omega)\rho \faa \omega\in\Omega\,.
\end{displaymath}
It was shown in \cite{Crauel_98_1} that $\mu_\omega$ is a Dirac measure concentrated on $a_\alpha(\omega)$, and linearizing along this invariant measure $\mu$ yields a negative Lyapunov exponent, given by
\begin{displaymath}
  \lambda_\alpha=-\frac{2}{\sigma^2}\int_\R(\alpha x-x^3)^2p_{\alpha,\sigma}(x)\,\rmd x\,.
\end{displaymath}
Moreover, the family $\set{a_\alpha(\omega)}_{\omega\in\Omega}$ is the global random attractor (see Appendix), which implies that the attractor bifurcation associated with a deterministic pitchfork bifurcation (that is, $K_\alpha$ bifurcates from a non-trivial object to a singleton) is destroyed by noise.

\subsection{Qualitative changes in uniform attractivity}\label{subsec2}

In order to establish qualitative changes in the attractivity of the unique random attracting fixed point $\set{a_\alpha(\omega)}_{\omega\in\Omega}$, a detailed understanding about the location of this attractor is needed. For this purpose, we use similar techniques as developed in \cite{Tearne_05_1,Tearne_08_1}.

\begin{proposition}\label{Prp1}
  Consider \eqref{Sde} for $\alpha\in\R$, and let $\set{a_\alpha(\omega)}_{\omega\in\Omega}$ be its unique random fixed point. Then for any $\eps>0$ and $T\geq 0$, there exists a measurable set $\cA \in \cF_{-\infty}^{T}$ (see Appendix) of positive measure such that
  \begin{displaymath}
    a_\alpha(\theta_s\omega)\in (-\eps,\eps)\fa  s\in [0,T] \mand \omega\in\cA \,.
  \end{displaymath}
\end{proposition}

\begin{proof}
  We first consider the case $\alpha\le 0$.
  According to \cite[Theorem~12]{Tearne_05_1}, there exists $\cA \in \cF_{-\infty}^{0}$ of positive measure such that $a(\omega)\in (-\eps/3,\eps/3)$ for all $\omega\in \cA$. Define
  \begin{displaymath}
    \cA^+:= \setB{\omega\in\Omega: \sup_{t\in[0,T]} |\omega(t)|\le \delta:= \frac{\eps}{2 e^{-\alpha T}}}\,.
  \end{displaymath}
  Then
  \begin{displaymath}
    |\vp(t,\omega,a(\omega)) - \phi(t,a(\omega))|\le \delta -\alpha \int_0^t |\vp(s,\omega,a(\omega)) - \phi(s,a(\omega))|\,\rmd s\,,
  \end{displaymath}
  where $\phi(\cdot,x_0)$ denotes the solution of the initial value problem
  \begin{displaymath}
    \dot x_t=\alpha x_t-x_t^3, \quad x(0)=x_0\,.
  \end{displaymath}
  Thus,
  \begin{displaymath}
   |\vp(t,\omega,a(\omega))| \le |a(\omega)| + \delta e^{-\alpha t} < \eps \fa t\in [0,T] \mand \omega \in \cA \cap \cA^+\,.
  \end{displaymath}

  This implies the assertion for $\alpha\le 0$. It remains to show the proposition for $\alpha>0$; the proof of this fact is divided in the following four steps.

  \emph{Step 1.} We will construct an absorbing set for the random dynamical system $(\theta, \vp)$. For this purpose, let $B_{\rho}(0)$ for some $\rho>0$ be a ball in $\R$, for which we will consider the pullback limit $\vp(t,\theta_{-t}\omega)B_{\rho}(0)$. Consider the Langevin equation
  \begin{equation}\label{Langevin}
    \rmd z =-\alpha z \,\rmd t+\sigma \rmd W(t)\,,
  \end{equation}
  and let $\psi:\R\times \Omega \times \R\to\R$ denote the associated random dynamical system, given by
  \begin{displaymath}
    \psi(t,\omega)z_0=e^{-\alpha t} z_0+\sigma\int_0^t e^{-\alpha(t-s)}\,\rmd W(s)\,.
  \end{displaymath}
  It follows that
  \begin{displaymath}
    \psi(t,\theta_{-t}\omega)z_0=e^{-\alpha t} z_0+\sigma\int_{-t}^0 e^{\alpha s}\,\rmd W(s)\,,
  \end{displaymath}
  which implies that
  \begin{displaymath}
  z(\omega):=\lim_{t\to\infty}    \psi(t,\theta_{-t}\omega)z_0= \sigma\int_{-\infty}^0 e^{\alpha s}\,\rmd W(s)
  \end{displaymath}
  is the unique random fixed point of \eqref{Langevin}.  Using the exponential martingale inequality, for almost all $\omega\in\Omega$, there are positive constants $A(\omega),B(\omega)$ such that
  \begin{equation}\label{IteratedLogaLaw}
    (\psi(t,\omega)z(\omega))^2\leq A(\omega)+ B(\omega)\ln (1+|t|) \fa t\in\R\,.
  \end{equation}
  Fix $\tau\ge0$ and $\omega_0\in\Omega$, and define $v(t):=\vp(t,\theta_{-\tau}\omega_0)x_0- \psi(t,\theta_{-\tau}\omega_0) z(\theta_{-\tau}\omega_0)$ for all $t\in\R$, where $x_0\in B_\rho(0)$. Using the integral form of \eqref{Sde}, we have
  \begin{align*}
    v(t)&=\vp(t,\theta_{-\tau}\omega_0)x_0-\psi(t,\theta_{-\tau}\omega_0)z(\theta_{-\tau}\omega_0)\\
    &=\int_0^t \opintb{\alpha \vp(s,\theta_{-\tau}\omega_0)x_0-(\vp(s,\theta_{-\tau}\omega_0)x_0)^3}\,\rmd s
    +
    \int_0^t \alpha \psi(s,\theta_{-\tau}\omega_0)z(\theta_{-\tau}\omega_0)\,\rmd s\,,
  \end{align*}
  which yields that
  \begin{equation}\label{Eq13}
    \dot v(t)+\alpha v(t)= 2\alpha(v(t)+\psi(t,\theta_{-\tau}\omega_0)z(\theta_{-\tau}\omega_0))-(v(t)+\psi(t,\theta_{-\tau}\omega_0)z(\theta_{-\tau}\omega_0))^3\,.
  \end{equation}
  Note that using Cauchy's Inequality, we obtain that for all $v,z\in\R$
  \begin{align*}
  3v^2z^2+\frac{\alpha^2}{3}
  &\geq
  2\alpha vz\,,\\
  \frac{v^4}{2}+2\alpha^2
  &\geq
  2\alpha v^2\,,\\
  \frac{v^4}{12}+   \frac{v^4}{12}+  \frac{v^4}{12}+\frac{3^{7}}{4}z^4
  &\geq
  3v^3z\,,\\
  \frac{v^4}{4}+\frac{3\sqrt[3] 3}{4}z^4+\frac{3\sqrt[3] 3}{4}z^4+\frac{3\sqrt[3] 3}{4}z^4
  &\geq
  3vz^3\,.
  \end{align*}
  Therefore,
  \begin{displaymath}
    \left( 2\alpha(v+z)-(v+z)^3\right)v
    \leq
    C (1+z^2+z^4) \fa  v,z\in\R\,,
  \end{displaymath}
  where $C:=\max\setB{\frac{7}{3}\alpha^2,\frac{3^7}{4}+\frac{3\sqrt[3] 3}{4}}$. Thus, from \eqref{Eq13} we derive that
  \begin{align*}
    v(t) \dot v(t) +\alpha v(t)^2
    &\leq C (1+(\psi(t,\theta_{-\tau}\omega_0)z(\theta_{-\tau}\omega_0))^2+(\psi(t,\theta_{-\tau}\omega_0)z(\theta_{-\tau}\omega_0))^4)\,,\\[1ex]
    &= C(1+z(\theta_{t-\tau}\omega_0)^2+z(\theta_{t-\tau}\omega_0)^4)
  \end{align*}
  which implies that
  \begin{align*}
    v(\tau)^2
    &\leq
     e^{-2\alpha \tau} v(0)^2+2C\int_0^\tau e^{2\alpha (s-\tau)} \left(1+ z(\theta_{s-\tau}\omega_0)^2+z(\theta_{s-\tau}\omega_0)^4\right)\,\rmd s\\[1ex]
     &\leq
     e^{-2\alpha \tau} v(0)^2+2C\int_{-\infty}^0 e^{2\alpha s} \left(1+z(\theta_{s}\omega_0)^2+z(\theta_{s}\omega_0)^4\right)\,\rmd s\,,
  \end{align*}
  where the existence of infinity integral follows from \eqref{IteratedLogaLaw}. Consequently,
  \begin{align*}
    (\vp(\tau,\theta_{-\tau}\omega_0)x_0)^2
    &\leq
    2 v(\tau)^2+2z(\omega_0)^2\\[1ex]
    &\leq
    2e^{-2\alpha \tau}v(0)^2+R(\omega_0)+2z(\omega_0)^2\\[1ex]
    &\leq
    4e^{-2\alpha \tau}(x_0^2+z(\theta_{-\tau}\omega_0)^2)+R(\omega_0)+2z(\omega_0)^2    ,
  \end{align*}
  where
  \begin{displaymath}
    R(\omega_0):=4C\int_{-\infty}^0 e^{2\alpha s}(1+z(\theta_s\omega_0)^2+z(\theta_s\omega_0)^4)\,\rmd s\,.
  \end{displaymath}
  Since $|x_0|<\rho$ and $\limsup_{\tau\to\infty} e^{-2\alpha \tau} |z(\theta_{-\tau}\omega_0)|=0$ it follows that $B_{R(\omega_0)+2 z(\omega_0)+1}(0)$ is an absorbing set of \eqref{Sde}. Thus, $a_\alpha(\omega)\in B_{R(\omega)+2z(\omega)+1}(0)$ for almost all $\omega\in\Omega$.

  \emph{Step 2.} In this step, we construct a measurable set $A_1\subset \cF_{-\infty}^0$ of a positive probability such that
  \begin{displaymath}
    a_\alpha(\omega)\in B_{1}(K) \fa \omega\in A_1\,.
  \end{displaymath}
  Define
  \begin{displaymath}
    A^-:=\{\omega: R(\omega)+2z(\omega)\leq \E[R+2z]\}.
  \end{displaymath}
  Clearly, $\mP(A^-)>0$ and we refer to \cite{Mao_97_1} for the existence of $\E[R+2z]$. Recall that $K$ denotes the global attractor for the deterministic case $\sigma = 0$. Then there exists $T_1>0$ such that
  \begin{equation}\label{Eq11}
    \phi\opintb{t,B_{\E [R+2z] +1}(0)}\subset B_{1/3}(K) \fa t\geq T_1\,,
  \end{equation}
  where $\phi(\cdot,x_0)$ denotes the solution of the initial value problem
  \begin{displaymath}
    \dot x_t=\alpha x_t-x_t^3, \quad x(0)=x_0\,.
  \end{displaymath}
  Let $\delta_1>0$ be a positive constant satisfying that
  \begin{equation}\label{Eq12}
    \delta_1\leq \frac{1}{9\sigma e^{\alpha T_1}}.
  \end{equation}
  Now, define $A^+:=\big\{\omega: \sup_{t\in [0,T_1]}|\omega(t)|\leq \delta_1\big\}$. From \cite[Section 6.8]{Ikeda_81_1}, the set $A^+$ has positive measure. Clearly, $A^-$ and $A^+$ are independent and therefore the set $A^-\cap A^+\in \cF_{-\infty}^{T_1}$ is also of positive probability measure. Choose and fix an arbitrary $\omega\in A^-\cap A^+$. By the definition of $A^-$, we get that $a_\alpha(\omega)\in B_{\E [R+2z]+1}(0)$. Since $a_\alpha(\omega)$ is a random fixed point of $\vp$, it follows that
  \begin{displaymath}
    a_\alpha(\theta_t\omega)=a_\alpha(\omega)+\int_0^t \opintb{\alpha a_\alpha(\theta_s\omega)-a_\alpha(\theta_s\omega)^3}\,\rmd s+\sigma \omega(t)\,.
  \end{displaymath}
  Define $u(t):=a_\alpha(\theta_t\omega)-\phi(t,a_\alpha(\omega))$. According to the definition of $\phi(t,\cdot)$, we obtain that
  \begin{displaymath}
  u(t)=\int_0^t \alpha u(s)-u(s)\left(a_\alpha(\theta_s\omega)^2+a_\alpha(\theta_s\omega)\phi(s,a_\alpha(\omega))+\phi(s,a_\alpha(\omega))^2\right)\,\rmd s+\sigma w(t),
  \end{displaymath}
  which together with the fact that $|w(t)|\leq \delta_1$ for all $t\in [0,T_1]$  implies that
  \begin{displaymath}
  |u(t)|\leq \sigma\delta_1+\int_0^t\alpha |u(s)|\,\rmd s\fa  t\in [0,T_1]\,.
  \end{displaymath}
  Using Gronwall's inequality, we get that
  \begin{displaymath}
  |u(t)|\leq \sigma\delta_1 e^{\alpha t}\leq \frac{1}{3}\fa t\in [0,T_1]\,.
  \end{displaymath}
  Therefore, by \eqref{Eq11} we get that $a_\alpha(\theta_{T_1}\omega)\in B_{1}(K)$. Consequently, the set $A_1:=\theta_{T_1}(A^-\cap A^+)$ satisfies the desired assertion in this step.

  \emph{Step 3.} In this step, we construct a measurable set $A_2\subset \cF_{-\infty}^0$ of a positive probability such that
  \begin{displaymath}
  a_\alpha(\omega)\in (-\delta_2,\delta_2)\fa \omega\in A_2\,,
  \end{displaymath}
  where $\delta_2:=\frac{\eps e^{-\alpha T}}{2(1+|\sigma|)}$. For this purpose, let $\eps_1\in \R_{>0}$ be arbitrary.
  According to the construction of the set $A$ in Step 2, we obtain that $a_\alpha(\omega)\in B_1(K)$ for all $\omega\in A$. This together with the fact that
  \begin{displaymath}
  B_1(K)=  B_1(K)\cup \bigcup_{n\in\Z} \big[n \eps_1,(n+1)\eps_1\big]
  \end{displaymath}
  implies that there exists $n\in\Z$ such that
  \begin{equation}\label{Eq14}
  |n|\eps_1\leq \sqrt\alpha+1 \qandq \mP \left(\Big\{\omega\in A: a_\alpha(\omega)\in \big[n \eps_1,(n+1)\eps_1\big]\Big\}\right)>0\,.
  \end{equation}
  We will now only deal with the case $n\geq 0$ and refer a similar treatment for the case $n<0$. Let $\widetilde \phi$ denote the solution of the following integral equation
  \begin{equation}\label{Eq15}
  x(t)=n\eps_1+\int_0^t \opintb{\alpha x(s)-x(s)^3} \rmd s -2(\alpha^{\frac{3}{2}}+1)t\,.
  \end{equation}
  Define $T_{\min}:=\min\setb{t\geq 0: \widetilde \phi(t)=0}$. We will show that $T_{\min}<1$. Suppose the contrary, i.e.\ $\widetilde \phi(t)>0$ for all $t\in [0,1]$. Using the inequality that $\alpha x-x^3\leq \alpha x$ for all $x\geq 0$ and \eqref{Eq15}, we get that
  \begin{displaymath}
  \widetilde \phi(t)\in (0,n\eps_1]\fa t\in [0,1]\,.
  \end{displaymath}
  Therefore,
  \begin{align*}
    \widetilde \phi(1)
    &=
    n\eps_1+\int_0^1 \opintb{\alpha \widetilde\phi(s)-\widetilde\phi(s)^3}\rmd s-2(\alpha^{\frac{3}{2}}+1)\\
    &\leq
    n\eps_1+\alpha n\eps_1-2(\alpha^{\frac{3}{2}}+1)<0\,,
  \end{align*}
  which leads to a contradiction. Now we define
  \begin{displaymath}
  \widetilde A_1:=\setb{\omega\in\Omega: \textstyle \sup_{t\in [0,T_{\min}]} \abvalb{\omega(t)+2\opintb{\alpha^{\frac{3}{2}}+1}t}<\eps_1}.
  \end{displaymath}
  Note that for any $\omega\in A_1\cap \widetilde A_1$, we have
  \begin{align*}
    a_\alpha(\theta_{T_{\min}}\omega)
    &=
    a_\alpha(\omega)+\int_0^t \opintb{\alpha a_\alpha(\theta_s\omega)-a_\alpha(\theta_s\omega)^3}\rmd s+\omega(t)\\
    &\leq
    2\eps_1+n\eps_1+\int_0^t \opintb{\alpha a_\alpha(\theta_s\omega)-a_\alpha(\theta_s\omega)^3}\rmd s-2(\alpha^{\frac{3}{2}}+1)t\,.
  \end{align*}
  Consequently, by choosing $\eps_1$ sufficiently small we get that $|a_\alpha(\theta_{T_{\min}}\omega)|<\delta_2$ for all $\omega\in \omega\in A_1\cap \widetilde A_1$. Thus, the set $A_2:=\theta_{T_{\min}}(A_1\cap \widetilde A_1)$ will satisfy the desired assertion in this step.

  \emph{Step 4.} Define
  \begin{displaymath}
    \widetilde A_2:=\setb{\omega\in\Omega: \textstyle\sup_{t\in [0,T]}|\omega(t)|\leq \delta_2}\,,
  \end{displaymath}
  where $\delta_2$ is defined as in Step 3.  Clearly, $A_2$ and $\widetilde A_2$ are independent and therefore the set $\mathcal A:=A_2\cap \widetilde A_2$ is also of positive probability measure. Choose and fix an arbitrary $\omega\in \mathcal A$. By the construction of $A_2$ as in Step 3, we get that $|a_\alpha(\omega)|<\delta_2$. Since $a_\alpha(\omega)$ is a random fixed point of $\vp$ it follows that
  \begin{displaymath}
    a_\alpha(\theta_t\omega)=a_\alpha(\omega)+\int_0^t \opintb{\alpha a_\alpha(\theta_s\omega)-a_\alpha(\theta_s\omega)^3}\,\rmd s+\sigma \omega(t)\,.
  \end{displaymath}
  which implies that
  \begin{displaymath}
  |a_\alpha(\theta_t\omega)|\leq (1+|\sigma|)\delta_2+\int_0^t\alpha |a_\alpha(\theta_s\omega)|\,\rmd s\fa t\in [0,T]\,.
  \end{displaymath}
  Using Gronwall's inequality, we get that
  \begin{displaymath}
  |a_\alpha(\theta_t\omega)|\leq (1+|\sigma|)\delta_2 e^{\alpha t}< \eps \fa t\in [0,T]\,.
  \end{displaymath}
  Thus, we get that $a_\alpha(\theta_t\omega)\in (-\eps,\eps)$ for all $t\in[0,T]$, which completes the proof.
\end{proof}

We now give a detailed description of the random bifurcation scenario for the stochastic differential equation \eqref{Sde} by means of both \emph{asymptotic} and \emph{finite-time} dynamical behaviour. The asymptotic description implies that there is a qualitative change in the uniformity of attraction of the unique random attractor $\set{a_\alpha(\omega)}_{\omega\in\Omega}$. On the other hand, the finite-time description shows that after the bifurcation, even if the time interval is very large, the (asymptotic) Lyapunov exponent cannot be observed with non-vanishing probability (by a finite-time Lyapunov exponent); however, before the bifurcation, the (asymptotic) Lyapunov exponent can be approximated by the finite-time Lyapunov exponent. Finite-time Lyapunov exponents for random dynamical systems have not been considered in the literature so far, but play an important role in the description of Lagrangian Coherent Structures in fluid dynamics \cite{Haller_01_1}.

Let $\set{a_\alpha(\omega)}_{\omega\in\Omega}$ denote the unique random attracting fixed point of a stochastic differential equation \eqref{Sde}. Then $\set{a_\alpha(\omega)}_{\omega\in\Omega}$ is called \emph{locally uniformly attractive} if there exists $\delta>0$ such that
\begin{displaymath}
  \lim_{t\to\infty}\sup_{x\in (-\delta,\delta)}\esssup_{\omega\in\Omega}|\vp(t,\omega)(a_\alpha(\omega)+x)-a_\alpha(\theta_t\omega)|=0.
\end{displaymath}

\begin{theorem}[Random pitchfork bifurcation, asymptotic description]\label{Theorem_1}
  Consider the stochastic differential equation \eqref{Sde} with the unique random attracting fixed point $\set{a_\alpha(\omega)}_{\omega\in\Omega}$. Then the following statements hold:
  \begin{itemize}
    \item [(i)]
    For $\alpha<0$, the random attractor $\set{a_\alpha(\omega)}_{\omega\in\Omega}$ is locally uniformly attractive; in fact, it is even globally uniformly exponential attractive, i.e.
    \begin{equation}\label{Eq1}
      |\vp(t,\omega,x)-\vp(t,\omega,a_\alpha(\omega))|\leq e^{\alpha t} |x-a_\alpha(\omega)|\fa x\in \R\,.
    \end{equation}
    \item [(ii)]
    For $\alpha>0$, the random attractor $\set{a_\alpha(\omega)}_{\omega\in\Omega}$ is not locally uniformly attractive.
  \end{itemize}
\end{theorem}

\begin{proof}
  (i) Let $x\in\R$ be arbitrary such that $x\not=a_\alpha(\omega)$. Using the monotonicity of solutions, we may assume that $\vp(t,\omega,x)>\vp(t,\omega,a_\alpha(\omega))$ for all $t\ge0$. The integral form of \eqref{Sde}\,,
  \begin{displaymath}
    \vp(t,\omega)x=x+\int_0^t \opintb{\alpha \vp(s,\omega)x-(\vp(s,\omega)x)^3}\,\rmd s+ \sigma \omega(t)
  \end{displaymath}
  yields that
  \begin{displaymath}
    \vp(t,\omega)x-\vp(t,\omega)a_\alpha(\omega)\le  x-a_\alpha(\omega)+\alpha \int_0^t \opintb{\vp(s,\omega)x-\vp(s,\omega)a_\alpha(\omega)}\,\rmd s\,.
  \end{displaymath}
  Using Gronwall's inequality implies \eqref{Eq1}, which finished this part of the proof.

  (ii) Suppose to the contrary that there exists $\delta>0$ such that
  \begin{equation*}
    \lim_{t\to\infty}\sup_{x\in (-\delta,\delta)} \esssup_{\omega\in\Omega}|\vp(t,\omega,a_\alpha(\omega)+x)-a_\alpha(\theta_t\omega)|=0\,,
  \end{equation*}
  which implies that there exists $N\in\N$ such that
  \begin{equation}\label{Eq2}
  \sup_{x\in (-\delta,\delta)} \esssup_{\omega\in\Omega}|\vp(t,\omega,a_\alpha(\omega)+x)-a_\alpha(\theta_t\omega)|
  < \frac{\sqrt \alpha}{4}\fa t\geq N\,.
  \end{equation}
  According to Proposition \ref{Prp1}, there exists $\mathcal A\in \cF_{-\infty}^0$ of positive probability such that $a_\alpha(\omega)\in (-\frac{\delta}{4},\frac{\delta}{4})$. Note that $-\sqrt\alpha$ and $\sqrt\alpha$ are two attractive fixed point for the deterministic differential equation
  \begin{displaymath}
    \dot x=\alpha x-x^3\,.
  \end{displaymath}
  Let $\phi(\cdot,x_0)$ denote the solution of the above deterministic equation which satisfies that $x(0)=x_0$. Then there exists $T>N$ such that
  \begin{equation}\label{Eq16}
    \phi(T,\delta/4)>\frac{\sqrt\alpha}{2} \qandq \phi(T,-\delta/4)<-\frac{\sqrt\alpha}{2}\,.
  \end{equation}
  For any $\eps>0$, we define
  \begin{displaymath}
    \mathcal A^+_\eps:=\setb{\omega\in\Omega: \textstyle \sup_{t\in [0,T]}|\omega(t)|<\eps}\,.
  \end{displaymath}
  Clearly, $\mathcal A^+_\eps\in \cF_{0}^T$ has positive probability, and thus, $\mP(\mathcal A\cap \mathcal A_{\eps}^+)=\mP(\mathcal A)\mP(\mathcal A_{\eps}^+)$ is positive. Due to compactness of $[0,T]$, there exists $\eps>0$ such that for all $\omega\in \cA^+_\eps$, we have
  \begin{displaymath}
    |\vp(T,\omega,\delta/4)- \phi(T,\delta/4)|\leq \frac{\sqrt\alpha}{4} \qandq |\vp(T,\omega,-\delta/4)- \phi(T,-\delta/4)|< \frac{\sqrt\alpha}{4}\,,
  \end{displaymath}
  which implies together with \eqref{Eq16} that
  \begin{displaymath}
  \vp(T,\omega,\delta/4)> \frac{\sqrt\alpha}{4} \qandq \vp(T,\omega,-\delta/4)< -\frac{\sqrt\alpha}{4}\,.
  \end{displaymath}
  Due to the fact that $|a_\alpha(\omega)|\leq \frac{\delta}{2}$ for all $\omega\in \mathcal A\cap \mathcal A_{\eps}^+$, we get that for all $\omega\in \mathcal A\cap \mathcal A_{\eps}^+$
  \begin{align*}
    &\sup_{x\in (-\delta,\delta)}|\vp(T,\omega,a_\alpha(\omega)+x)-a_\alpha(\theta_T\omega)|\\
    \geq&
    \max\setb{\vp(T,\omega,\delta/4)-a_\alpha(\theta_T\omega)|,|\vp(T,\omega,-\delta/4)-a_\alpha(\theta_T\omega)|}\,.
  \end{align*}
  Consequently,
  \begin{displaymath}
  \sup_{x\in (-\delta,\delta)} \esssup_{\omega\in\Omega}|\vp(t,\omega,a_\alpha(\omega)+x)-a_\alpha(\theta_t\omega)|> \frac{\sqrt\alpha}{4} \,,
  \end{displaymath}
  which contradicts to \eqref{Eq2} and the proof is complete.
\end{proof}

For the description of the bifurcation via finite-time properties, consider a compact time interval $I=[0,T]$ and define the corresponding \emph{finite-time Lyapunov exponent} associated with the invariant measure $a_\alpha(\omega)$ by
\begin{displaymath}
  \lambda^{T,\omega}_\alpha:=\frac{1}{T}\ln\abvallr{\frac{\partial \vp_\alpha}{\partial x}(T,\omega,a_{\alpha}(\omega))}\,.
\end{displaymath}
Clearly, the (classical) Lyapunov exponent $\lambda_\alpha^\infty$ associated with the random fixed point $a_\alpha(\omega)$ is given by
\begin{displaymath}
  \lambda_\alpha^\infty=\lim_{T\to\infty}\lambda^{T,\omega}_\alpha\,.
\end{displaymath}
In contrast to classical Lyapunov exponent, the finite-time Lyapunov exponent is, in general, a non-constant random variable.
\begin{theorem}[Random pitchfork bifurcation, finite-time description]\label{Finite-timeBifurcation}
  Consider the stochastic differential equation \eqref{Sde} with the unique random attracting fixed point $\set{a_\alpha(\omega)}_{\omega\in\Omega}$. For any finite time interval $[0,T]$, let $\lambda_\alpha^{T,\omega}$ denote the finite-time Lyapunov exponent associated with $\set{a_\alpha(\omega)}_{\omega\in\Omega}$. Then the following statements hold:
  \begin{itemize}
    \item [(i)]
    For $\alpha<0$, the random attractor $\set{a_\alpha(\omega)}_{\omega\in\Omega}$ is finite-time attractive, i.e.
    \begin{displaymath}
    \lambda_\alpha^{T,\omega}\leq \alpha<0 \fa \omega\in\Omega\,.
    \end{displaymath}
    \item [(ii)]
    For $\alpha>0$, the random attractor $\set{a_\alpha(\omega)}_{\omega\in\Omega}$ is not finite-time attractive, i.e.
    \begin{displaymath}
    \mP\big\{\omega\in\Omega: \lambda_\alpha^{T,\omega}>0\big\}>0.
    \end{displaymath}
  \end{itemize}
\end{theorem}
\begin{proof}
  (i) follows directly from \theoref{Theorem_1} (i).

  (ii) Choose $\eps:=\frac{\sqrt\alpha}{2}>0$. According to Proposition \ref{Prp1}, there exists a measurable set $\cA\in \cF_{-\infty}^T$ of positive probability such that
  \begin{displaymath}
    a_\alpha(\theta_s\omega)\in (-\eps,\eps)\fa s\in [0,T]\,.
  \end{displaymath}
  Let $\omega\in\cA$ be arbitrary and we will estimate $\lambda_\alpha^{T,\omega}$. Note that the linearized equation along the random fixed point $a_\alpha(\omega)$ is given by
  \begin{displaymath}
    \dot \xi(t)=(\alpha-3a_\alpha(\theta_t\omega)^2)\xi(t)\,.
  \end{displaymath}
  We thus get
  \begin{displaymath}
    \lambda_{\alpha}^{T,\omega}
    =
    \alpha-\frac{1}{T}\int_0^T 3 a_\alpha(\theta_t\omega)^2\rmd t
    \geq
    \frac{\alpha}{4}\,,
  \end{displaymath}
  which completes the proof.
\end{proof}


This theorem implies that the change in the signature of finite-time Lyapunov exponents indicates a qualitative change in the dynamics. This means that the bifurcation is observable in practice, since finite-time Lyapunov exponents are numerically computable quantities. Note that the numerical approximation of classical Lyapunov exponents is difficult in general. In the special case of random matrix products with positive matrices, however, \cite{Pollicott_10} established explicit bounds for the numerical approximation of (classical) Lyapunov exponents recently.

\subsection{The dichotomy spectrum at the bifurcation point}\label{subsec3}

We will compute the dichotomy spectrum of the linearization around the unique random attracting fixed point $\set{a_\alpha(\omega)}$ of the system \eqref{Sde}. As a direct consequence, we observe that hyperbolicity is lost at the bifurcated point $\alpha=0$.

\begin{theorem}\label{DichotomySpetrumc-Bifurcation}
  Let $\Phi_\alpha(t,\omega):=\frac{\partial\phi_\alpha}{\partial x}(t,\omega,a_\alpha(\omega))$ denote the linearized random dynamical system along the random fixed point $a_\alpha(\omega)$. Then the dichotomy spectrum $\Sigma_\alpha$ of $\Phi_\alpha$ is given by
  \begin{displaymath}
  \Sigma_\alpha=[-\infty,\alpha]\fa \alpha\in \R\,.
  \end{displaymath}
\end{theorem}
\begin{proof}
From the linearized equation along $a_\alpha(\omega)$
\begin{displaymath}
  \dot\xi(t)=(\alpha-3a_\alpha(\theta_t\omega)^2)\xi(t)\,,
\end{displaymath}
we derive that
\begin{equation}\label{Eq17}
  \Phi_\alpha(t,\omega)=\exp\left(\int_0^t \big(\alpha-3a_\alpha(\theta_s\omega)^2\big)\rmd s\right)\,.
\end{equation}
Consequently,
\begin{displaymath}
|\Phi_\alpha(t,\omega)|\leq e^{\alpha |t|}\fa t\in\R\,,
\end{displaymath}
which implies that $\Sigma_\alpha\subset (-\infty,\alpha]$. Thus, it is sufficient to show that $(-\infty,\alpha]\subset \Sigma_\alpha$. For this purpose, let $\gamma\in (-\infty,\alpha]$ be arbitrary. Suppose the opposite that $\Phi_\alpha$ admits an exponential dichotomy with growth rate $\gamma$ with an invariant projection $P_\gamma$ and positive constants $K,\eps$. We now consider two cases: (i) $P_\gamma=\id$ and (ii) $P_\gamma=0$:

\emph{Case (i).} $P_\gamma=\id$, i.e.~we have
\begin{equation}\label{Eq18}
  \Phi_\alpha(t,\omega)\leq K e^{(\gamma-\eps) t}\fa t\geq 0\,.
\end{equation}
Choose and fix $T>0$ such that $e^{\frac{\eps }{4}T}>K$. According to Proposition \ref{Prp1}, there exists a measurable set $\mathcal A\subset \cF_{-\infty}^T$ of positive measure such that
\begin{displaymath}
  a_\alpha(\theta_s\omega)\in \opintb{-\sqrt\eps /2, \sqrt\eps / 2} \fa \omega\in\mathcal A \mand s\in [0,T]\,.
\end{displaymath}
From \eqref{Eq17} we derive that
\begin{displaymath}
  |\Phi_\alpha(T,\omega)| \geq  e^{T\left(\alpha-\frac{3\eps}{4}\right)}> K e^{(\gamma-\eps)T}\,,
\end{displaymath}
which leads to a contradiction to \eqref{Eq18}.

\emph{Case (ii)}: $P_\gamma=0$, i.e.~we have
\begin{equation*}
  \Phi_\alpha(t,\omega)\geq \frac{1}{K} e^{(\gamma-\eps) t}\fa  t\geq 0\,,
\end{equation*}
which together with \eqref{Eq17} implies that
\begin{equation}\label{Eq18b}
  \frac{\ln K+(\alpha-\gamma)t}{3}\geq \int_0^t a_\alpha(\theta_s\omega)^2 \,\rmd s\,.
\end{equation}
Choose and fix $T>0$ such that
\begin{displaymath}
  \frac{(T-1)^3}{3}>\frac{\ln K+(\alpha-\gamma)T}{3}\,.
\end{displaymath}
Consider the following integral equation
\begin{displaymath}
  x(t)=\int_0^t \opintb{\alpha x(s)-x(s)^3} \rmd s +\frac{t^4}{4}-\alpha\frac{t^2}{2}+t\,.
\end{displaymath}
Clearly, the explicit solution of the above equation is $x(t)=t$. Due to the compactness, there exists $\eps>0$ such that for any $x(0)\in(-\eps,\eps)$ and $\omega(t)$ with $\sup_{t\in [0,T]}|\omega(t)-\frac{t^4}{4}+\alpha\frac{t^3}{3}-t|\leq \eps$ then the solution $x(t)$ of the following equation
\begin{displaymath}
  x(t)=x(0)+\int_0^t (\alpha x(s)-x(s)^3)\,\rmd s+\omega(t)
\end{displaymath}
satisfies that $\sup_{t\in [0,T]}|x(t)-t|\leq 1$. According to Proposition \ref{Prp1}, there exists a measurable set $\mathcal A^{-}_\eps\subset \cF_{-\infty}^0$ of positive measure such that $a_\alpha(\omega)\in (-\eps,\eps)$ for all $\omega\in\mathcal A$. Define $\mathcal A_+\subset \cF_{0}^T$ by
\begin{displaymath}
  \mathcal A^+_\eps:=\setb{\omega\in \cF_{0}^T: \textstyle \sup_{t\in [0,T]}|\omega(t)-t^4/4+\alpha t^2/2-t|\leq \eps}\,.
\end{displaymath}
Therefore, for all $\omega\in \mathcal A^{-}_\eps\cap\mathcal A^+_\eps$, we get
\begin{displaymath}
  \sup_{t\in [0,T]}|a_\alpha(\theta_t\omega)-t|\leq 1\,,
\end{displaymath}
which implies that
\begin{displaymath}
  \int_0^T a_\alpha(\theta_s\omega)^2\,\rmd s \ge \frac{(T-1)^3}{3} > \frac{\ln K+(\alpha-\gamma)T}{3}\,,
\end{displaymath}
which leads to a contradiction to \eqref{Eq18b}. The proof is complete.
\end{proof}

We have seen in \theoref{Finite-timeBifurcation} that the bifurcation of \eqref{Sde} manifests itself also via finite-time Lyapunov exponents: before the bifurcation, all finite-time Lyapunov exponents are negative, and after the bifurcation, one observes positive finite-time Lyapunov exponents with positive probability. This implies in particular that the set of all Lyapunov exponents observed almost surely within a \emph{finite time} does not converge to the (asymptotic) Lyapunov exponents when time tends to infinity. The following theorem makes it precise that in contrast to asymptotic Lyapunov exponents, the dichotomy spectrum includes limits of the set of finite-time Lyapunov exponents.

\begin{theorem}\label{theo_1}
  Let $(\theta, \Phi)$ be a linear random dynamical system on $\R^d$ with dichotomy spectrum $\Sigma$. Define the finite-time Lyapunov exponent
  \begin{displaymath}
    \lambda(T,\omega,x) := \frac{1}{T} \ln \frac{\norm{\Phi(t,\omega)x}}{\norm{x}} \fa T>0,\,\omega\in\Omega\mand x\in\R^d\setminus\set{0}\,.
  \end{displaymath}
  Then
  \begin{displaymath}
    \lim_{T\to\infty} \esssup_{\omega\in\Omega} \!\!\!\!\sup_{x\in\R^d\setminus\set{0}} \!\!\!\!\!\!\lambda(T,\omega, x) = \sup \Sigma
  \end{displaymath}
  provided that  $\sup \Sigma<\infty$ and
  \begin{displaymath}
    \qandq \lim_{T\to\infty} \essinf_{\omega\in\Omega} \!\!\!\!\inf_{x\in\R^d\setminus\set{0}} \!\!\!\!\!\!\lambda(T,\omega, x) = \inf \Sigma
  \end{displaymath}
  provided that  $\inf \Sigma>-\infty$.
\end{theorem}

\begin{proof}
  By definition of $\lambda(T,\omega,x)$, we get that for all $T,S\geq 0$
  \begin{displaymath}
    (T+S)\esssup_{\omega\in\Omega} \!\!\!\!\sup_{x\in\R^d\setminus\set{0}} \!\!\!\!\!\!\lambda(T+S,\omega, x)\leq T\esssup_{\omega\in\Omega} \!\!\!\!\sup_{x\in\R^d\setminus\set{0}} \!\!\!\!\!\!\lambda(T,\omega, x)+S\esssup_{\omega\in\Omega} \!\!\!\!\sup_{x\in\R^d\setminus\set{0}} \!\!\!\!\!\!\lambda(S,\omega, x).
  \end{displaymath}
  This implies that the sequence $(T\esssup_{\omega\in\Omega} \sup_{x\in\R^d\setminus\set{0}} \lambda(T,\omega, x))_{T\geq 0}$ is subadditive. We thus obtain
  \begin{displaymath}
    \lim_{T\to\infty} \esssup_{\omega\in\Omega} \!\!\!\!\sup_{x\in\R^d\setminus\set{0}} \!\!\!\!\!\!\lambda(T,\omega, x) =     \limsup_{T\to\infty} \esssup_{\omega\in\Omega} \!\!\!\!\sup_{x\in\R^d\setminus\set{0}} \!\!\!\!\!\!\lambda(T,\omega, x).
  \end{displaymath}
  We first prove that provided $\sup \Sigma<\infty$, we have
  \[
    \gamma:=\limsup_{T\to\infty} \esssup_{\omega\in\Omega} \sup_{x\in\R^d\setminus\set{0}} \lambda(T,\omega, x)=\sup \Sigma.
  \]
  Since $\sup \Sigma<\infty$ it follows that there exists $K>0$ such that
  \begin{equation}\label{Eq19}
  \norm{\Phi(t,\omega)x}\le K e^{t\sup\Sigma} \fa t\ge0\,.
  \end{equation}
  Assume first that $\gamma < \sup \Sigma$. This means that there exists a $t_0>0$ such that for all $t\ge  t_0$ and for almost all $\omega\in\Omega$, we have $\norm{\Phi(t,\omega)x}\le e^{t/2(\gamma+\sup\Sigma)}$. Thus, together with \eqref{Eq19}, we obtain for all $t\ge  0$ and for almost all $\omega\in\Omega$ that
  \begin{displaymath}
  \norm{\Phi(t,\omega)x}\le \widehat K e^{t/2(\gamma+\sup\Sigma)},\qquad \widehat K:=\max\{1,Ke^{t_0/2(\sup\Sigma-\gamma)}\}.
  \end{displaymath}
  Hence, $\sup\Sigma\leq \gamma+\sup\Sigma$, which is a contradiction. Assume now that $\gamma > \sup \Sigma$. This means in particular that $\sup\Sigma < \infty$. Hence, there exists a $K>0$ such that for almost all $\omega\in\Omega$, we have
  \begin{displaymath}
    \norm{\Phi(t,\omega)x} \le K e^{t/2(\gamma+\sup\Sigma)}\norm{x}\fa x\in\R^d\,.
  \end{displaymath}
  This leads to $\lambda(t,\omega,x)\le (\gamma+\sup\Sigma)/2$ for all $x\in\R^d\setminus\{0\}$, and thus,
  \[
    \gamma=\limsup_{T\to\infty} \esssup_{\omega\in\Omega} \sup_{x\in\R^d\setminus\set{0}} \lambda(T,\omega, x)\le (\gamma+\sup\Sigma)/2,
  \]
  which proves the first equality. Similarly, one can show that
  \begin{displaymath}
     \lim_{T\to\infty} \essinf_{\omega\in\Omega} \!\!\!\!\inf_{x\in\R^d\setminus\set{0}} \!\!\!\!\!\!\lambda(T,\omega, x) = \inf \Sigma
  \end{displaymath}
  provided that  $\inf \Sigma>-\infty$, which finishes the proof of this theorem.
\end{proof}

In the following example, we construct explicitly a linear random dynamical system with $\sup \Sigma=\infty$ but
  \begin{displaymath}
    \lim_{T\to\infty} \esssup_{\omega\in\Omega} \!\!\!\!\sup_{x\in\R^d\setminus\set{0}} \!\!\!\!\!\!\lambda(T,\omega, x) <\infty.
  \end{displaymath}
An example of a linear random dynamical system with $\inf \Sigma=-\infty$ but
\begin{displaymath}
  \lim_{T\to\infty} \essinf_{\omega\in\Omega} \!\!\!\!\inf_{x\in\R^d\setminus\set{0}} \!\!\!\!\!\!\lambda(T,\omega, x) >-\infty.
\end{displaymath}
can be constructed analogously. This example shows the importance of the assumption $\sup \Sigma<\infty$ or $\inf \Sigma>-\infty$ in the above theorem.

\begin{example}
  Similarly to \examref{Example1}, there exist infinitely many measurable sets $\set{U_n}_{n\in\N}$ of positive measure such that  $U_n,\theta U_n,\theta^2U_n$ for $n\in \N$ are pairwise disjoint. We define a random mapping $A:\Omega\rightarrow \R$ as follows:
  \[
  A(\omega)=\left\{
              \begin{array}{c@{\;:\;}l}
                \frac{1}{n} &  \omega\in U_n\cup \theta^2 U_n\,, n\in\N\,,\\[1.5ex]
                n &  \omega\in \theta U_n \,, n\in\N\,,\\[1.5ex]
                1 & \hbox{otherwise}.
              \end{array}
            \right.
  \]
  Let $\Phi$ denote the discrete-time random dynamical system generated by $A$. Since $\log\|A(\cdot)\|$ is neither bounded from above nor from below, we get that $\Sigma(\Phi)=(-\infty,\infty)$. On the other hand, it is easy to see that for all $T\geq 2$ we get that
  \[
  \esssup_{\omega\in\Omega}\log|\Phi(T,\omega)|=1,
  \]
  which implies that
  \[
  \lim_{T\to\infty}\esssup_{\omega\in\Omega}\frac{1}{T}\log|\Phi(T,\omega)|=0.
  \]
\end{example}

\section{Topological equivalence of random dynamical systems }\label{sec_1}

This section deals with topological equivalence of random dynamical system \cite{Imkeller_01_2,Imkeller_02_1,Li_05_1,Arnold_98_1}. This concept has not been used so far to study bifurcations of random dynamical systems, and the main aim of this section is to discuss topological equivalence for the stochastic differential equation \eqref{Sde} from \secref{sec_bif}, given by
\begin{displaymath}
  \rmd x= \opintb{\alpha x-x^3}\rmd t+ \sigma \rmd W_t\,.
\end{displaymath}

The concept of topological equivalence for random dynamical systems \cite[Definition 9.2.1]{Arnold_98_1} differs from the corresponding deterministic notion of topological equivalence in the sense that instead of one homeomorphism (mapping orbits to orbits), the random version is given by a family of homeomorphisms $\set{h_\omega}_{\omega\in\Omega}$. The precise definition is given as follows.

\begin{definition}[Topological equivalence]\label{deftopeq}
  Let $(\Omega,\cF,\mP)$ be a probability space,  $\theta:\T\times \Omega\to\Omega$ a metric dynamical system and $(X_1,d_1)$, $(X_2,d_2)$ be metric spaces. Then two random dynamical systems $(\vp_1:\T\times\Omega\times X_1\to X_1,\theta)$ and $(\vp_2:\T\times\Omega\times X_1\to X_1,\theta)$ are called \emph{topologically equivalent} if there exists a conjugacy $h:\Omega\times X_1\to X_2$ fulfilling the following properties:
  \begin{itemize}
    \item[(i)] For almost all $\omega\in\Omega$, the function $x \mapsto h(\omega,x)$ is a homeomorphism from $X_1$ to $X_2$.
    \item[(ii)] The mappings $\omega \mapsto h(\omega, x_1)$ and $\omega \mapsto h^{-1}(\omega, x_2)$ are measurable for all $x_1 \in X_1$ and $x_2 \in X_2$.
    \item[(iii)] The random dynamical systems $\vp_1$ and $\vp_2$ are \emph{cohomologous}, i.e.
    \begin{displaymath}
      \vp_2(t, \omega, h(\omega, x)) = h(\theta_t\omega, \vp_1(t, \omega, x)) \fa x\in X_1 \mbox{ and almost all } \omega\in\Omega\,.
    \end{displaymath}
  \end{itemize}
\end{definition}

A bifurcation is then described by means of a lack of topological equivalence at the bifurcation point. The following theorem says that near the bifurcation point $\alpha = 0$, all systems of \eqref{Sde} are equivalent.

\begin{theorem}\label{Theorem2}
  Let $(\theta:\R\times \Omega\to\Omega, \vp_\alpha:\R\times\Omega\times\R\to \R)$ denote the random dynamical system generated by system \eqref{Sde}. Then there exists an $\eps>0$ such that for all $\alpha\in(-\eps,\eps)$, the random dynamical systems $\vp_\alpha$ are topologically equivalent to the dynamical system $(e^{-t}x)_{t,x\in\R}$, i.e.~there exists a conjugacy $h:\Omega\times \R\to\R$ such that for almost all $\omega\in\Omega$, we have
  \begin{displaymath}
  \vp_\alpha(t,\omega,h(\omega,x))=h(\theta_t\omega,e^{-t}x)\fa t,x\in\R\,.
  \end{displaymath}
\end{theorem}

\begin{proof}
  Let $a_\alpha(\omega)$ denote the unique random fixed point of \eqref{Sde}. According to the results in \cite{Crauel_98_1}, we obtain that
  \begin{displaymath}
    \E a_\alpha(\omega)^2= \frac{\int_{-\infty}^\infty u^2 \exp\left(\frac{1}{\sigma^2}\opintb{\alpha u^2-\frac{1}{2}u^4}\right)\,\rmd u}{\int_{-\infty}^\infty\exp\left(\frac{1}{\sigma^2}\opintb{\alpha u^2-\frac{1}{2}u^4}\right)\,\rmd u}\,.
  \end{displaymath}
  Therefore,
  \begin{displaymath}
    \lim_{\alpha\to 0} \E a_\alpha(\omega)^2= \frac{\int_{-\infty}^\infty u^2 \exp\opintb{{-}\frac{u^4}{2\sigma^2}}\,\rmd u}{\int_{-\infty}^\infty \exp\opintb{{-}\frac{u^4}{2\sigma^2}}\,\rmd u}\,.
  \end{displaymath}
  Then there exists an $\eps>0$ such that for all $\alpha\in(-\eps,\eps)$, we have
  \begin{displaymath}
    \delta:=\frac{\E a_\alpha(\omega)^2}{4}-\alpha>0\,.
  \end{displaymath}
  For any $x\in\R$ and $(t,\omega)\in\R\times\Omega$, we define
  \begin{equation}\label{Eq4}
    \psi(t,\omega,x):=\vp_\alpha(t,\omega,x+a_\alpha(\omega))-a_\alpha(\theta_t\omega)\,.
  \end{equation}
  By using the transformation function $g(\omega,x):=x-a_\alpha(\omega)$, the random dynamical systems $\vp_\alpha$ and $\psi$ are topologically equivalent. Hence, it is sufficient to show that $\psi$ is topologically equivalent to the dynamical system $(e^{-t}x)_{t,x\in\R}$. We first summarise some properties of $\psi$:
  \begin{itemize}
    \item [(i)]
    Since $a_\alpha(\omega)$ is a random fixed point of $\vp_\alpha$, it follows that
    \begin{displaymath}
      \psi(t,\omega,0)=0 \fa t\in\R \mand \omega\in\Omega\,.
    \end{displaymath}
    \item [(ii)]
    Due to the monotonicity of $\vp_\alpha$, for $x_1>x_2$, we have
    \begin{displaymath}
      \psi(t,\omega,x_1)>\psi(t,\omega,x_2) \fa  t\in\R \mand \omega\in\Omega\,.
    \end{displaymath}
    \item [(iii)]
    From \eqref{Sde}, we derive that
    \begin{align*}
    \psi(t,\omega,x)
    &=
    x+\int_0^t \psi(s,\omega,x) \opintb{\alpha-a_\alpha(\theta_s\omega)^2-a_\alpha(\theta_s\omega)\vp_\alpha(s,\omega,a_\alpha(\omega)+x)-\\
    &\qquad\qquad\qquad\qquad\qquad\vp_\alpha(s,\omega,a_\alpha(\omega)+x)^2}\,\rmd s\,.
    \end{align*}
    Consequently,
    \begin{align*}
      \psi(t,\omega,x)
      &=
      x\exp\opintBg{\int_0^t\alpha-a_\alpha(\theta_s\omega)^2-a_\alpha(\theta_s\omega)\vp_\alpha(s,\omega,a_\alpha(\omega)+x)\\
      &
      \qquad\qquad\qquad\qquad\qquad-\vp_\alpha(s,\omega,a_\alpha(\omega)+x)^2\,\rmd s}\,.
    \end{align*}
  \end{itemize}
  According to Birkhoff's ergodic theorem, there exists an invariant set $\widetilde \Omega$ of full measure such that
  \begin{equation}\label{Eq3}
    \lim_{t\to\pm\infty}\frac{1}{t}\int_0^t a_\alpha(\theta_s\omega)^2\,\rmd s=\E a_\alpha(\omega)^2\,.
  \end{equation}
  Choose and fix $\omega\in\widetilde\Omega$. From \eqref{Eq3}, there exists $T>0$ such that for all $|t|>T$ we have
  \begin{displaymath}
    \left|\frac{1}{t}\int_0^t a_\alpha(\theta_s\omega)^2\,\rmd s-\E a_\alpha(\omega)^2\right|\leq \frac{\delta}{2}\,.
  \end{displaymath}
  In what follows, we will show the following estimates on $\psi(t,\omega,x)$ for $x>0$:
  \begin{itemize}
  \item [(iii)] For $t\geq T$, we get
  \begin{displaymath}
    \psi(t,\omega,x)
    \le
     x\exp\opintlr{\int_0^t \alpha-\frac{a_\alpha(\theta_s\omega)^2}{4}\,\rmd s}\leq  e^{-\delta t/2}x.
  \end{displaymath}
  \item [(iv)] For $t\leq -T$, we get
  \begin{displaymath}
    \psi(t,\omega,x)
    \geq
     x\exp\opintlr{\int_0^t \alpha-\frac{a_\alpha(\theta_s\omega)^2}{4}\,\rmd s}
     \geq e^{-\delta t/2}x\,.
  \end{displaymath}
  \end{itemize}
  Consequently, we get that
  \begin{displaymath}
  \lim_{r\to\infty}\int_r^\infty \psi(s,\omega,x)\,\rmd s=0
  \qandq
  \lim_{r\to-\infty}\int_r^\infty \psi(s,\omega,x)\,\rmd s=\infty.
  \end{displaymath}
  Hence, there exists a unique $r(\omega,x)$ such that
  \begin{equation}\label{Eq5}
  \int_{r(\omega,x)}^\infty \psi(s,\omega,x)\,\rmd s=1.
  \end{equation}
  Similarly, $r(\omega,x)$ for $x<0$ is defined to satisfy
  \begin{equation}\label{Eq6}
  \int_{r(\omega,x)}^\infty \psi(s,\omega,x)\,\rmd s=-1.
  \end{equation}
  Using the cocycle property of $\psi$, we obtain that
  \begin{equation}\label{Eq7}
  r(\omega,x)=r(\theta_s\omega,\psi(s,\omega,x))+s.
  \end{equation}
  Define a function
  \begin{displaymath}
  g(\omega,x)
  =
  \left\{
    \begin{array}{c@{\;:\;}l}
      e^{r(\omega,x)} & x>0\,, \\
      0 & x=0\,, \\
      -e^{r(\omega,x)} & x<0\,.
    \end{array}
  \right.
  \end{displaymath}
  We will now show that $g$ transforms the random dynamical system $\psi$ to the dynamical system $(e^{-t}x)_{t,x\in\R}$:
  \begin{itemize}
  \item [(i)] For any $x>0$, we have $\psi(s,\omega,x)>0$ and thus from the definition of the function $g$ it follows that
  \begin{displaymath}
  g(\theta_s\omega,\psi(s,\omega,x))
  =
  e^{r(\theta_s\omega,\psi(s,\omega,x))},
  \end{displaymath}
  which implies together with \eqref{Eq7} that
  \begin{displaymath}
  g(\theta_s\omega,\psi(s,\omega,x))=e^{r(\omega,x)-s}=e^{-s}\psi(s,\omega,x)\,.
  \end{displaymath}
  Similarly, for $x<0$ we also have $g(\theta_s\omega,\psi(s,\omega,x))=e^{-s}\psi(s,\omega,x)$ for all $s\in\R,\omega\in\Omega$.
  \item [(ii)] Choose and fix $\omega\in\widetilde \Omega$. We will show that $g_\omega:\R\rightarrow \R, x\mapsto g(\omega,x)$ is a homeomorphism.\\
  \emph{Injectivity}: From the definition of $g$, it is easily seen that for $x_1>0>x_2$ we have
  \begin{displaymath}
  g_\omega(x_1)>0>g_\omega(x_2).
  \end{displaymath}
  On the other hand, based on strict monotonicity of $\psi$ we get that for $x_1>x_2>0$
  \begin{displaymath}
  \int_{r(\omega,x_2)}^\infty \psi(s,\omega,x_1)\,\rmd s>\int_{r(\omega,x_2)}^\infty \psi(s,\omega,x_2)\,\rmd s=1.
  \end{displaymath}
  Consequently, $r(\omega,x_1)
  >r(\omega,x_2)$ and thus $g_\omega(x_1)>g_\omega(x_2)$. Similarly, for $0>x_1>x_2$ we also have $g_\omega(x_1)>g_\omega(x_2)$. Therefore, $g_\omega$ is strictly increasing and thus injective.\\
  \emph{Continuity}: We first show that $\lim_{x\to 0+}g_\omega(x)=0$. Let $\eps>0$ be arbitrary. Choose $\tilde{T}>T$ such that $\frac{2}{\delta}e^{-\frac{\delta \tilde{T}}{2}}<\frac{1}{3}, e^{-\tilde{T}}<\eps,$ and for all $t\geq \tilde{T}$ we get
  \begin{displaymath}
    \psi(t,\omega,x)\leq e^{-\frac{\delta t}{2}}x\,.
  \end{displaymath}
  As a consequence, for all $x\in (0,1)$ we get
  \begin{equation}\label{Eq8}
    \int_{\tilde{T}}^\infty \psi(s,\omega,x)\,\rmd s\leq \int_{\tilde{T}}^\infty e^{-\frac{\delta s}{2}}\,\rmd s<\frac{\eps}{3}\,.
  \end{equation}
  Since $\lim_{x\to 0} \psi(s,\omega,x)=0$ and $\clintb{{-}\tilde{T},\tilde{T}}$ is a compact interval, there exists $\delta^*$ such that
  \begin{displaymath}
  \int_{-\tilde{T}}^{\tilde{T}} \psi(s,\omega,\delta^*)\,\rmd s<\frac{\eps}{3},
  \end{displaymath}
  which together with \eqref{Eq8} implies that
  \begin{displaymath}
  \int_{-\tilde{T}}^\infty \psi(s,\omega,x)\,\rmd s<\frac{2}{3}\fa  x\in \big(0,\min(1,\delta^*)\big)\,.
  \end{displaymath}
  Therefore, $r(\omega,x)<-\tilde{T}$ and thus $g_\omega(x)<\eps$ for all $x\in \big(0,\min(1,\delta^*)\big)$. Hence, $\lim_{x\to 0+}g_\omega(x)=0$ and similarly we also have $\lim_{x\to 0-}g_\omega(x)=0$ and thus $g_\omega$ is continuous at $0$. The continuity of $g$ on the whole real line can be proved in a similar way.\\
  \emph{Surjectivity}: It is easy to prove surjectivity from
  \begin{displaymath}
  \lim_{x\to\infty} g_\omega(x)=\infty \qandq\lim_{x\to-\infty} g_\omega(x)=-\infty\,.
  \end{displaymath}
  \end{itemize}
  This finishes the proof of this theorem.
\end{proof}

This theorem implies that the stochastic differential equation \eqref{Sde} does not admit a bifurcation at $\alpha=0$ which is induced by the above concept of topological equivalence. In addition, because of the observations in \theoref{DichotomySpetrumc-Bifurcation}, this concept of equivalence is not in correspondence with the dichotomy spectrum (linear systems which are hyperbolic and non-hyperbolic can be equivalent).

We will show now that the concept of a \emph{uniform topological equivalence} is the right tool to obtain the bifurcations studied in this paper.

\begin{definition}[Uniform topological equivalence]
  Let $(\Omega,\cF,\mP)$ be a probability space,  $\theta:\T\times \Omega\to\Omega$ a metric dynamical system and $(X_1,d_1)$, $(X_2,d_2)$ be metric spaces. Then two random dynamical systems $(\vp_1:\T\times\Omega\times X_1\to X_1,\theta)$ and $(\vp_2:\T\times\Omega\times X_1\to X_1,\theta)$ are called \emph{uniformly topologically equivalent} with respect to a random fixed point $\set{a_\alpha(\omega)}_{\omega\in\Omega}$ of $\vp_1$ if there exists a conjugacy $h:\Omega\times X_1\to X_2$ fulfilling the following properties:
  \begin{itemize}
    \item[(i)] For almost all $\omega\in\Omega$, the function $x \mapsto h(\omega,x)$ is a homeomorphism from $X_1$ to $X_2$.
    \item[(ii)] The mappings $\omega \mapsto h(\omega, x_1)$ and $\omega \mapsto h^{-1}(\omega, x_2)$ are measurable for all $x_1 \in X_1$ and $x_2 \in X_2$.
    \item[(iii)] The random dynamical systems $\vp_1$ and $\vp_2$ are \emph{cohomologous}, i.e.
    \begin{displaymath}
      \vp_2(t, \omega, h(\omega, x)) = h(\theta_t\omega, \vp_1(t, \omega, x)) \fa x\in X_1 \mbox{ and almost all } \omega\in\Omega\,.
    \end{displaymath}
    \item[(iv)]
    We have
    \begin{displaymath}
      \lim_{\delta\to0} \esssup_{\omega\in\Omega} \sup_{x\in B_\delta(a_\alpha(\omega))} d_2(h(\omega, x), h(\omega, a_\alpha(\omega))) = 0
    \end{displaymath}
    and
    \begin{displaymath}
      \lim_{\delta\to0} \esssup_{\omega\in\Omega} \sup_{x\in B_\delta(h(\omega, a_\alpha(\omega)))} d_1(h^{-1}(\omega, x), a_\alpha(\omega)) = 0\,.
    \end{displaymath}
  \end{itemize}
\end{definition}

Note that, in comparison to the concept of topological equivalence (\defref{deftopeq}), we added (iv) to take uniformity into account.

We show now that uniform topological equivalence preserves local uniform attractivity.

\begin{proposition}\label{prop_1}
  Let $(\Omega,\cF,\mP)$ be a probability space,  $\theta:\T\times \Omega\to\Omega$ a metric dynamical system and $(X_1,d_1)$, $(X_2,d_2)$ be metric spaces, and let $(\vp_1:\T\times\Omega\times X_1\to X_1,\theta)$ and $(\vp_2:\T\times\Omega\times X_2\to X_2,\theta)$ be two random dynamical systems which are \emph{uniformly topologically equivalent} with respect to a random fixed point $\set{a_\alpha(\omega)}_{\omega\in\Omega}$ of $\vp_1$. Let $h:\Omega\times X_1\to X_2$ denote the conjugacy. Then $\set{a_\alpha(\omega)}_{\omega\in\Omega}$ is locally uniformly attractive for $\vp_1$ if and only if $\set{h(\omega, a_\alpha(\omega))}_{\omega\in\Omega}$ is locally uniformly attractive for $\vp_2$.
\end{proposition}

\begin{proof}
  Suppose that $\set{a_\alpha(\omega)}_{\omega\in\Omega}$ is locally uniformly attractive for $\vp_1$ and let $\eta>0$. Then there exists a $\gamma>0$ such that
  \begin{displaymath}
    \esssup_{\omega\in\Omega} \sup_{x\in B_\gamma(a_\alpha(\omega))} d_2(h(\omega, x), h(\omega,a_\alpha(\omega))) \le \eta\,.
  \end{displaymath}
  Since $\set{a_\alpha(\omega)}_{\omega\in\Omega}$ is locally uniformly attractive for $\vp_1$, there exists a $\delta>0$ and a $T>0$ such that
  \begin{displaymath}
    \esssup_{\omega\in\Omega} \sup_{x\in B_\delta(a_\alpha(\omega))} d_1(\vp_1(t,\omega, x), a_\alpha(\theta_t\omega)) \le \frac{\gamma}{2} \fa t\ge T\,.
  \end{displaymath}
  Hence, for all $t\ge T$, we have
  \begin{displaymath}
    \esssup_{\omega\in\Omega} \sup_{x\in B_\delta(a_\alpha(\omega))} d_2(h(\theta_t\omega,\vp_1(t,\omega,x)), h(\theta_t\omega,a_\alpha(\theta_t\omega))) \le \eta\,.
  \end{displaymath}
  This means that for all $t\ge T$, we have
  \begin{displaymath}
    \esssup_{\omega\in\Omega} \sup_{x\in B_\delta(a_\alpha(\omega))} d_2(\vp_2(t,\omega, h(\omega, x)), h(\theta_t\omega,a_\alpha(\omega))) \le \eta\,,
  \end{displaymath}
  and there exists a $\beta>0$ such that
  \begin{displaymath}
    \esssup_{\omega\in\Omega} \sup_{x\in B_\beta(h(\omega, a_\alpha(\omega)))} d_1(h^{-1}(\omega, x), a_\alpha(\omega)) \le \frac{\delta}{2}\,.
  \end{displaymath}
  Finally, this means that for all $t\ge T$, we have
  \begin{displaymath}
    \esssup_{\omega\in\Omega} \sup_{x\in B_\beta(h(\omega,a_\alpha(\omega)))} d_2(\vp_2(t,\omega, x), h(\theta_t\omega,a_\alpha(\omega))) \le \eta\,,
  \end{displaymath}
  which finishes the proof that $\set{h(\omega, a_\alpha(\omega))}_{\omega\in\Omega}$ is locally uniformly attractive for $\vp_2$.
\end{proof}

As a corollary to this proposition, it follows that \eqref{Sde} admits a bifurcation.

\begin{theorem}\label{theo1}
  The stochastic differential equation \eqref{Sde} admits a random bifurcation at $\alpha=0$ which is induced by the concept of uniform topological equivalence.
\end{theorem}

\begin{proof}
  This is a direct consequence of \theoref{Theorem_1} and \propref{prop_1}.
\end{proof}

\medskip

{\Large\centerline{\textbf{\underbar{Appendix}}}}

\bigskip

\emph{Metric dynamical systems.} Let $\cB(Y)$ denote the Borel $\sigma$-algebra of a metric space $Y$. Consider a time set $\T=\R$ or $\T =\Z$, and let $(\Omega,\cF,\mP)$ be a probability space. A $(\cB(\T)\otimes \cF, \cF)$-measurable function $\theta:\T\times \Omega \to \Omega$ is called a \emph{measurable dynamical system} if $\theta(0,\omega) =\omega$ and $\theta(t+s,\omega) =  \theta(t, \theta(s, \omega))$ for all $t,s\in\T$ and $\omega\in\Omega$. We use the abbreviation $\theta_t\omega$ for $\theta(t,\omega)$. A measurable dynamical system is said to be \emph{measure preserving} or \emph{metric} if $\mP\theta(t,A) = \mP A$ for all $t\in\T$ and $A\in\cF$, and such a dynamical system is called \emph{ergodic} if for any $A\in\cF$ satisfying $\theta_t A = A$ for all $t\in\T$, one has $\mP A\in\set{0,1}$. A particular metric dynamical system, which naturally is used when dealing with (one-dimensional) stochastic differential equations, is generated by the Brownian motion. More precisely, $\Omega:= C_0(\R,\R):=\set{\omega\in C(\R,\R): \omega(0)=0}$. Let $\Omega$ be equipped with the compact-open topology and the Borel $\sigma$-algebra $\cF:=\cB(C_0(\R, \R))$. Let $\mP$ denote the Wiener probability measure on $(\Omega, \cF)$. The metric dynamical system is then given by the Wiener shift $\theta : \R\times \Omega\to\Omega$, defined by $\theta(t, \omega(\cdot)):= \omega(\cdot+t)-\omega(t)$, and it is well-known that $\theta$ is ergodic \cite{Arnold_98_1}. On $(\Omega,\cF)$, we have the natural filtration
\begin{displaymath}
  \cF_s^t:= \sigma\opintb{\omega(u)-\omega(v): s\le u,v\le t} \fa s\le t\,,
\end{displaymath}
with $\theta_u^{-1} \cF_s^t= \cF_{s+u}^{t+u}$.

\emph{Invariant measures.} For a given random dynamical system ($\theta,\vp$), let $\Theta:\T\times \Omega\times X \to \Omega\times X$ denote the corresponding \emph{skew product flow}, given by $\Theta(t,\omega,x):= (\theta_t \omega, \vp(t,\omega)x)$. This is a measurable dynamical system on the extended phase space $\Omega\times X$. A probability measure $\mu$ on $(\Omega\times X,\cF\otimes \cB)$ is said to be an \emph{invariant measure} if
\begin{itemize}
  \item[(i)] $\mu(\Theta_t A)=\mu(A)$ for all $t\in\T$ and $A\in \cF\otimes\cB$,
  \item[(ii)] $\pi_\Omega \mu = \mP$,
\end{itemize}
where $\pi_\Omega \mu$ denotes the marginal of $\mu$ on $(\Omega,\cF)$. If the metric space $X$ is a Polish space, i.e., it is separable and complete, then an invariant measure $\mu$ admits a $\mP$-almost surely unique \emph{disintegration} \cite[Proposition 1.4.3]{Arnold_98_1}, that is a family of probability measures $(\mu_\omega)_{\omega\in\Omega}$ with
\begin{displaymath}
  \mu(A) = \int_\Omega\int_X \1_A(\omega,x) \,\rmd \mu_\omega(x) \,\rmd \mP(\omega)\,.
\end{displaymath}

\emph{Random sets.} A function $\omega \mapsto M(\omega)$ taking values in the subsets of the phase space $X$ of a random dynamical system is called a \emph{random set} if $\omega \mapsto d(x,M(\omega))$ is measurable for each $x \in X$, and we use the term $\omega$-fiber of $M$ for the set $M(\omega)$. We call $M$ \emph{closed} or \emph{compact} if all $\omega$-fibers are closed or compact, respectively. A random set $M$ is called \emph{invariant} with respect to the random dynamical system $(\theta,\vp)$ if $\vp(t,\omega)M(\omega) = M(\theta_t\omega)$ for all $t\in \R$ and $\omega\in\Omega$.

\emph{Random attractors.} A nonempty, compact and invariant random set $\omega\mapsto A(\omega)$ is called \emph{global random attractor} for a random dynamical system $(\theta,\vp)$ with metric state space $(X,d)$, if it attracts all bounded sets in the sense of pullback attraction, i.e., for all bounded sets $B\subset X$, one has
\begin{displaymath}
  \lim_{t\to\infty} \dist(\vp(t,\theta_{-t}\omega)B, A(\omega)) = 0 \faa \omega\in \Omega\,,
\end{displaymath}
where $\dist (C,D):= \sup_{c\in C} d(c,D)$ is the \emph{Hausdorff semi-distance} of $C$ and $D$. A global random attractor (given it exists) is always unique \cite{Crauel_94_1}. The existence of random attractors is proved via so-called absorbing sets \cite{Flandoli_96_1}. A bounded set $B\subset X$ is called \emph{absorbing set} if for almost all $\omega\in\Omega$ and any bounded set $D\subset X$, there exists a time $T>0$ such that
\begin{displaymath}
  \vp(t,\theta_{-t}\omega)D \subset B \fa t\ge T\,.
\end{displaymath}
Given an absorbing set $B$, it follows that there exists a global random attractor $\set{A(\omega)}_{\omega\in\Omega}$, given by
\begin{displaymath}
  A(\omega):=\bigcap_{\tau\ge 0} \overline{\bigcup_{t\ge \tau} \vp(t,\theta_{-t} \omega)B} \faa \omega\in\Omega\,.
\end{displaymath}

\emph{Lyapunov exponents and Multiplicative Ergodic Theory.} Given a linear random dynamical system $(\theta,\Phi)$ in $\R^d$, a \emph{Lyapunov exponent} is given by
\begin{displaymath}
  \lambda=\lim_{t\to\pm\infty}\frac{1}{|t|}\ln\norm{\Phi(t,\omega)x} \fs \omega\in\Omega\mand x\in\R^d\,.
\end{displaymath}
The Multiplicative Ergodic Theorem \cite{Oseledets_68_1} shows that there are only finitely many Lyapunov exponents provided the random dynamical system is ergodic and fulfills an integrability condition. More precisely, consider a linear random dynamical system $(\theta:\T\times \Omega\to\Omega, \Phi:\T \times \Omega \to \R^{d\times d})$, suppose that $\theta$ is ergodic and $\Phi$ satisfies the integrability condition
\begin{displaymath}
  \sup_{t\in[0,1]} \ln^+\opintb{\norm{\Phi(t,\cdot)^{\pm 1}}}\in L^1(\mP)\,,
\end{displaymath}
here $\ln^+(x):= \max\set{0,\ln(x)}$. Then the Multiplicative Ergodic Theorem states that almost surely, there exist at most $d$ Lyapunov exponents $\lambda_1<\lambda_2<\dots<\lambda_p$ and fiber-wise decomposition
\begin{displaymath}
  \R^d=O_1(\omega)\oplus O_2(\omega)\oplus\dots\oplus O_p(\omega) \faa \omega\in\Omega
\end{displaymath}
into Oseledets subspaces $O_i\subset \R^d$ such that for all $i \in\set{1,\dots,p}$ and almost all $\omega\in\Omega$, one has
\begin{displaymath}
  \lim_{t\to\pm\infty}\frac{1}{|t|}\ln\norm{\Phi(t,\omega)x}=\lambda_i \fa 0\not= x\in O_i(\omega)\,.
\end{displaymath}


\providecommand{\bysame}{\leavevmode\hbox to3em{\hrulefill}\thinspace}
\providecommand{\MR}{\relax\ifhmode\unskip\space\fi MR }
\providecommand{\MRhref}[2]{%
  \href{http://www.ams.org/mathscinet-getitem?mr=#1}{#2}
}
\providecommand{\href}[2]{#2}

\end{document}